\documentclass[11pt]{amsart}

\date{version of \today}
\usepackage{amsmath}
\usepackage{amssymb}
\usepackage{amsxtra}
\usepackage{abstract}
\usepackage{mathrsfs}
\usepackage{mathtools}
\usepackage[all]{xy}
\usepackage{amscd}
\usepackage{amsthm}
\usepackage[dvips]{graphicx}
\usepackage{ulem}
\usepackage{color}
\usepackage{appendix}
\usepackage{tikz}
\usepackage{tikz-cd}
\usepackage{url}

\newtheorem{Thm}{Theorem}[section]
\newtheorem{Lem}[Thm]{Lemma}
\newtheorem{Def}[Thm]{Definition}
\newtheorem{Cor}[Thm]{Corollary}
\newtheorem{Prop}[Thm]{Proposition}
\newtheorem{Ex1}[Thm]{Example}
\newtheorem{Rem1}[Thm]{Remark}
\newtheorem{Conj}[Thm]{Conjecture}
\newtheorem{Prob}[Thm]{Problem}

\newenvironment{Rem}{\begin{Rem1}\rm}{\end{Rem1}}
\newenvironment{Ex}{\begin{Ex1}\rm}{\end{Ex1}}

\usepackage{bbm}
\usepackage{enumitem}

\tikzset{every picture/.style={line width=0.75pt}} 

\newcommand{\rad}{\mathop{\mathrm{rad}}}
\newcommand{\soc}{\mathop{\mathrm{soc}}}

\newenvironment{theorem}{\begin{Thm}}{\end{Thm}}

\newenvironment{definition}{\begin{Def}}{\end{Def}}

\newenvironment{proposition}{\begin{Prop}}{\end{Prop}}
\newenvironment{example}{\begin{Ex1}}{\end{Ex1}}

\usepackage[bookmarks=true, colorlinks=true, citecolor=blue, linkcolor=black]{hyperref}

\newcommand{\qa}{kQ/I}

\newcommand{\m}{m}

\newcommand{\lra}{\longrightarrow}

\newcommand{\ra}{\rightarrow}
\newcommand{\sdp}{\times\kern-.2em\vrule height1.1ex depth-.05ex}
\newcommand{\epi}{\lra \kern-.8em\ra}

 \normalbaselines
\begin{document}

	\title{Classification of Brauer graph algebras under stable equivalence of Morita type}
	\author{Nengqun Li$^a$, Pengyun Chen$^b$, Yuming Liu$^{b,*}$, and Bohan Xing$^b$}
	\maketitle
	
	\renewcommand{\thefootnote}{\alph{footnote}}
	\setcounter{footnote}{-1} 
	\renewcommand{\thefootnote}{\alph{footnote}}
	\setcounter{footnote}{-1} 
	\setcounter{footnote}{-1} \footnote{$^a$Nengqun Li, School of Mathematics, Liaoning Normal University,
		Dalian 116029,  P. R. China.}
    \setcounter{footnote}{-1} \footnote{$^b$Pengyun Chen, Yuming Liu, and Bohan Xing, School of Mathematical Sciences, Laboratory of Mathematics and Complex Systems, Beijing Normal University,
		Beijing 100875,  P. R. China.}
	\setcounter{footnote}{-1} \footnote{E-mail addresses: pychen@mail.bnu.edu.cn (P. Chen); linengqun@lnnu.edu.cn (N. Li); ymliu@bnu.edu.cn (Y. Liu); bhxing@mail.bnu.edu.cn (B. Xing).}
	\setcounter{footnote}{-1} \footnote{$^*$Corresponding author.}
	
	{\noindent\small{\bf Abstract:} The classification of Brauer graph algebras under derived equivalence was recently given by Opper and Zvonareva. In this paper, we prove that two Brauer graph algebras are derived equivalent if and only if they are stably equivalent of Morita type. As an application, we show that every stable Picard group orbit of simple-images of Morita type contains a liftable representative. Moreover, we give a new proof of the fact that Brauer graph algebras are closed up to semisimple summands under stable equivalence of Morita type.}

\section{Introduction}
	
Brauer graph algebras form one of the most important classes of tame symmetric algebras and have attracted increasing attention since their introduction in \cite{DF}. On the one hand, Brauer graph algebras are generalization of Brauer tree algebras, the latter ones describe all group algebras of finite representation type up to Morita equivalence. This connection between Brauer graph algebras and modular group representation theory has remained a subject of continuous interest (see for example \cite{DF,DR}). On the other hand,  Brauer graph algebras coincide with symmetric special biserial algebras (\cite{Ro,Sch}) and the indecomposable modules over a Brauer graph algebra can be classified in terms of string and band modules. From this connection the representation types and the Auslander-Reiten quivers of Brauer graph algebras have been extensively studied (\cite{ES,BS,D}). For other aspects of Brauer graph algebras, such as their connections with gentle algebras and the surface cluster theory, we refer the reader to the survey article \cite{S}.

A direction of concern in recent years on study Brauer graph algebras is to describe algebras which are stably equivalent or derived equivalent to Brauer graph algebras. In this direction, Antipov and Zvonareva \cite{AZ2} proved that Brauer graph algebras are closed under derived equivalence, and Opper and Zvonareva \cite{OZ} gave a complete classification of Brauer graph algebras up to derived equivalence. However, for stable equivalences of Brauer graph algebras, much less is known. Antipov and Zvonareva \cite{AZ1} proved that a Brauer graph algebra can be stably equivalent to a stably biserial algebra only (here we assume that the algebras considered have no summand isomorphic to a Nakayama algebra of radical square zero) and that the Auslander-Reiten conjecture is true for algebras stably equivalent to Brauer graph algebras. Note that in characteristic different from 2, Brauer graph algebras coincide with symmetric stably biserial algebras, but in general there do exist symmetric stably biserial algebras which are not isomorphic to Brauer graph algebras.

The question of whether Brauer graph algebras are closed under stable equivalence remains open. The authors of the present paper gave a partial answer to this question in \cite{CLLX}. In particular, we proved that Brauer graph algebras are closed up to semisimple summands under stable equivalence of Morita type (\cite[Theorem 4.4]{CLLX}). In the present paper, we classify Brauer graph algebras under stable equivalence of Morita type. Our main result is as follows.

\begin{Thm}\textnormal{(see Theorem~\ref{cor-stm-der-BGA})}\label{main-result}
Let $A$ and $A'$ be Brauer graph algebras over an algebraically closed field $k$, associated with Brauer graphs $\Gamma$ and $\Gamma'$ respectively. Then $A$ and $A'$ are stably equivalent of Morita type if and only if they are derived equivalent. In particular, if $A$ and $A'$ are stably equivalent of Morita type such that neither $A$ nor $A'$ is local, then we have the following statements:
	\begin{enumerate}
		\item $\Gamma$ and $\Gamma'$ have same number of vertices, edges and faces;
		\item $\Gamma$ and $\Gamma'$ have same multi-sets of perimeters of faces;
		\item $\Gamma$ and $\Gamma'$ have same multi-sets of multiplicities of vertices;
		\item either both or none of $\Gamma$ and $\Gamma'$ are bipartite.
	\end{enumerate}
\end{Thm}

Note that the invariants in $(1)$-$(4)$ form a complete derived invariants for Brauer graph algebras (see \cite{OZ}) apart from some exceptional cases. Our strategy is to show that each of them is preserved under stable equivalences of Morita type apart from some exceptional cases. As part of this analysis, over a field of characteristic zero, we obtain explicit descriptions of the projective and stable centers of Brauer graph algebras; see Lemma~\ref{lem-proj-center} and Proposition~\ref{stable center of BGA}.

Beyond the classification problem, it is natural to ask whether a given stable equivalence of Morita type can be lifted to a derived equivalence. Theorem~\ref{main-result} does not mean that a stable equivalence of Morita type between two
Brauer graph algebras can be lifted to a derived equivalence. In fact, \cite[Example~6.4]{LL} provides
a stable auto-equivalence of Morita type over some Brauer graph algebra that is not lifted to
a derived equivalence. Following \cite{CKL}, the liftability question can be studied through simple-images of Morita type. In \cite{CKL}, it is asked whether every simple-image of Morita type over a self-injective algebra is liftable? It is only known that this question has a positive answer for the class of representation-finite self-injective algebras. This motivates the following weaker lifting problem, where $\operatorname{StPic}(A)$ denotes the stable Picard group of $A$.

\begin{Prob}\label{prob:liftable-orbits}
Let $A$ be a self-injective algebra. Does every $\operatorname{StPic}(A)$-orbit of simple-images of Morita type contain a liftable representative?
\end{Prob}

In general, this problem has a negative answer (see \cite[Remark~2.10(3)]{CKL}). Our main theorem shows that Problem \ref{prob:liftable-orbits} has a positive
answer for the class of Brauer graph algebras.

\begin{Prop}\textnormal{(see Proposition~\ref{prop:orbit-contains-liftable} and \ref{prop:classification-of-orbits})}
Let $A=A_{\Gamma}$ be a Brauer graph algebra.
\begin{enumerate}
\item Every $\operatorname{StPic}(A)$-orbit of simple-images of Morita type contains a liftable representative.
\item If $\Gamma$ is neither a loop whose unique vertex has multiplicity one nor an edge whose two vertices both have multiplicity two, then these orbits are in bijection with the isomorphism classes of Brauer graphs $\Gamma'$ such that $A_{\Gamma'}$ is derived equivalent to $A_{\Gamma}$.
\end{enumerate}
\end{Prop}

Finally, we mention that using some more subtle invariants, we give a new proof in Appendix \ref{sec-appendix} for the fact that Brauer graph algebras are closed up to semisimple summands under stable equivalences of Morita type.

The paper is organized as follows. In Section \ref{sec:prelim} we recall Brauer graph algebras and some invariants under stable equivalence of Morita type. In Section \ref{sec:classification-BGA-stM} we establish the invariance of multi-sets of multiplicities of vertices of Brauer graphs under stable equivalence of Morita type by using some invariants. Based on this, we classify Brauer graph algebras under stable equivalence of Morita type. In Section \ref{sec:rmk} we study simple-images of Morita type over Brauer graph algebras. In Appendix \ref{sec-appendix} we give another proof of the result that Brauer graph algebras are closed up to semisimple summands under stable equivalence of Morita type.

\section*{Data availability} The datasets generated during the current study are available from the corresponding author on reasonable request.

\section{Preliminaries}\label{sec:prelim}

	In this paper we fix $k$ to be an algebraically closed field. All algebras are finite-dimensional $k$-algebras. Unless explicitly stated otherwise, all modules considered in this paper are finitely generated left modules. For an algebra $A$ we denote by $\mathrm{rad}(A)$ and $\mathrm{soc}(A)$ the radical and the socle of $A$ respectively. Furthermore, let $A=\qa$ be a bounded quiver algebra, which means $Q$ is a finite quiver and $I$ is an admissible ideal in $kQ$.  We denote by $s(p)$ the source vertex of a path $p$ and by $t(p)$ its terminus vertex. We will write paths from right to left, for example, $p=\alpha_{n}\alpha_{n-1}\cdots\alpha_{1}$ is a path with starting arrow $\alpha_{1}$ and ending arrow $\alpha_{n}$. A path is called a cycle if $s(p)=t(p)$.
	By abuse of notation we sometimes view an element in $kQ$ as an element in $\qa$ if no confusion can arise.

	Recall that for a finite-dimensional algebra $A$,
	\begin{itemize}
		\item the derived category $\mathcal{D}(A)$ is the localization of homotopy category $\mathcal{K}(A)$ by inverting quasi-isomorphisms;
		\item the (left) stable category $ A\text{-}\underline{\mathrm{mod}}$ is the quotient of (left) module category $ A\text{-}{\mathrm{mod}}$ modulo the ideal of maps that factor through projective modules.
	\end{itemize}
	We say that two algebras are derived equivalent (resp. stably equivalent) if their derived (resp. stable) categories are equivalent.

	\subsection{Brauer graph algebras}
	\

	\subsubsection{Basic definitions}\label{subsec:def-BGA}
	\
	
	In this section, we recall the basic knowledge about Brauer graph algebras. For the convenience of the following construction, we use ribbon graphs to define these algebras.
	
	\begin{Def}\textnormal{(\cite[Definition 1.1]{OZ})}\label{ribbon-graph}
		A ribbon graph is a tuple $\Gamma=(V,H,r,\iota,\rho)$, where
		\begin{enumerate}
			\item $V$ is a finite set whose elements are called vertices;
			
			\item $H$ is a finite set whose elements are called half-edges;
			
			\item $r: H\rightarrow V$ is a function;
			
			\item $\iota: H\rightarrow H$ is an involution without fixed points;
			
			\item $\rho: H\rightarrow H$ is a permutation whose cycles correspond to the sets $H_v:=r^{-1}(v)$, $v\in V$.
			
		\end{enumerate}
	\end{Def}
	
	Therefore, every ribbon graph defines a graph with vertex set $V$ whose edges are the orbits of $\iota$. An edge $\{h, \iota(h)\}$ is incident to the vertices $r(h)$ and $r(\iota(h))$. In fact, every ribbon graph gives rise to an oriented surface (see for example in \cite[Section 1.1]{OZ}).
	
		If $\Gamma$ is a ribbon graph, we write $V(\Gamma)$ and $H(\Gamma)$ for its sets of vertices and half-edges as well as  $H_v(\Gamma)$ for all half-edges $h\in H(\Gamma)$ such that $r(h)=v\in V(\Gamma)$. To simplify the notation, we often write $h^{\pm}:=\rho^{\pm 1}(h)$ for the successor and predecessor of a half-edge $h$ and $\bar{h}$ for its associated edge. The set of all edges is denoted by $E(\Gamma)$. Denote by $val(v)$ the valency of the vertex $v\in V$. It is defined to be the number of edges in $G$ incident to $v$, with the convention that a loop is counted twice.	
	Unless stated otherwise, we will assume that $\Gamma$ is connected, which means its underlying graph is connected.
	
		\begin{Def}\textnormal{(\cite[Definition 1.5]{OZ})}\label{BG}
		A Brauer graph is a pair $(\Gamma,m)$ consisting of a ribbon graph $\Gamma$ and a function $m: V(\Gamma)\rightarrow \mathbb{Z}_+$.
	\end{Def}
	
	The function $m$ in Definition \ref{BG} is referred to as the multiplicity function and its values as multiplicities. Frequently, we omit $m$ from the notation and refer to $\Gamma$ as a Brauer graph. We denote by $\textbf{n}$ any constant multiplicity function with value $n$ at each vertex and say that a Brauer graph $(\Gamma,m)$ is multiplicity-free if $m=\textbf{1}$. In particular,  we call a given vertex $v$ truncated if $\m(v)val(v)=1$.

\begin{Def}\textnormal{(see \cite[Definition 1.4]{OZ})}\label{iso-of-BGs}
Two Brauer graphs $(\Gamma,m)$ and $(\Gamma',m')$ are said to be isomorphic if there exists a pair $(\phi_V,\phi_H)$ consisting of bijections $\phi_V:V(\Gamma)\rightarrow V(\Gamma')$ and $\phi_H:H(\Gamma)\rightarrow H(\Gamma')$ which commute with the functions $r,\iota,\rho$ and preserve multiplicities. That is, for each $v\in V(\Gamma)$ and $h\in H(\Gamma)$, $r(\phi_H(h))=\phi_V(r(h))$, $\iota(\phi_H(h))=\phi_H(\iota(h))$, $\rho(\phi_H(h))=\phi_H(\rho(h))$ and $m(v)=m'(\phi_V(v))$.
\end{Def}
	
	To any Brauer graph $(\Gamma,m)$ one can associate a quiver $Q=Q_\Gamma$ and an admissible ideal of relations $I=I_\Gamma$ in the path algebra $kQ$. If $\Gamma$ is a single edge with two truncated vertices and $m=\textbf{1}$, then we define $Q$ to be a loop $\alpha$ and $I$ to be the ideal generated by $\alpha^2$. Otherwise, we define $Q$ and $I$ as follows.
	
	\begin{enumerate}
		\item The vertices of $Q$ correspond to the edges of $\Gamma$ and for every $h\in H$ with $r(h)$ not truncated, there is an arrow $\alpha_h:\bar{h}\rightarrow\bar{h^+}$. The assignment $\alpha_h\mapsto \alpha_{h^-}$ defines a permutation $\sigma=\sigma_\Gamma$ of the arrows of $Q$ whose orbits are in bijection with vertices which are not truncated. Hence every arrow $\alpha$ defines a closed path
	$$C_\alpha=\alpha\sigma(\alpha)\cdots\sigma^l(\alpha)$$
	where $l+1$ denotes the cardinality of the $\sigma$-orbit of $\alpha$. Every vertex of $Q$ is the starting point of at most two cycles of form $C_\alpha$. If $\alpha=\alpha_h$, set $\m(C_\alpha):=\m(r(h))$.
	
	\item The ideal $I_\Gamma$ is generated by the following set of relations:
	\begin{enumerate}
		\item $$C_\alpha^{\m(C_\alpha)}=C_\beta^{\m(C_\beta)},$$
		where $\alpha,\beta\in Q_1$ and $t(\alpha)=t(\beta)$, that is $\alpha$ and $\beta$ end at same edge of $\Gamma$.
		
		\item $$\alpha C_{\sigma(\alpha)}^{\m(C_\alpha)}=0,$$
		where $\alpha$ is an arbitrary arrow in $Q$.
		
		\item $$\alpha\beta=0,$$
		where $\alpha,\beta\in Q_1$ are composable and $\sigma(\alpha)\neq \beta$;
	\end{enumerate}
	\end{enumerate}
	
	The resulting finite-dimensional algebra $kQ_\Gamma/I_\Gamma$ will be denoted by $A_\Gamma$. Note that every nonzero path of $A_\Gamma$ is a subpath of a cycle $C_\alpha^{\m(C_\alpha)}$ for some arrow $\alpha$. We call a given $\alpha\in Q_1$ an arrow around $v\in V$ if $\alpha=\alpha_h$ with $r(h)=v$.
	
	\begin{Def}\textnormal{(\cite[Definition 1.6]{OZ})}
		A $k$-algebra $A$ is called a Brauer graph algebra if there exists a Brauer graph $(\Gamma,\m)$ such that $A\cong A_\Gamma$ as $k$-algebras.
	\end{Def}

Note that isomorphic Brauer graphs give rise to isomorphic Brauer graph algebras. The following result shows that the associated Brauer graph of a Brauer graph algebra is unique in most cases.

\begin{Prop}\textnormal{(\cite[Lemma 3.1]{AZ2})}\label{prop-uniqueness-of-BG}
Let $(\Gamma,\m)$ and $(\Gamma',\m')$ be two Brauer graphs such that the associated Brauer graph algebras $A_{\Gamma}$ and $A_{\Gamma'}$ are isomorphic. If $(\Gamma,m)$ is not a loop with $1$ as the multiplicity of the unique vertex, or an edge with $2$ as the multiplicity of both vertices, then the Brauer graphs $(\Gamma,\m)$ and $(\Gamma',\m')$ are isomorphic.
\end{Prop}

The following example shows that two non-isomorphic Brauer graphs may define isomorphic Brauer graph algebras.

\begin{Ex1}\label{ex-isomorphic-BGAs}
Suppose that $\mathrm{char}(k)\neq 2$, $\Gamma$ is an edge with $2$ as the multiplicity of both vertices, and $\Gamma'$ is a loop with $1$ as the multiplicity of the unique vertex. Then $A_{\Gamma}\cong k[X,Y]/(X^2-Y^2,XY)$ and $A_{\Gamma'}\cong k[T,U]/(T^2,U^2)$. The map $f:A_{\Gamma}\rightarrow A_{\Gamma'}$, $X\mapsto T+U$, $Y\mapsto \sqrt{-1}(T-U)$ gives an $k$-algebra isomorphism between $A_{\Gamma}$ and $A_{\Gamma'}$.
\end{Ex1}

\begin{Lem}\label{lem-not-st-eq}
Let $\mathrm{char}(k)=2$ and let $\Gamma$ and $\Gamma'$ be two Brauer graphs, where $\Gamma$ is an edge with $2$ as the multiplicity of both vertices, and $\Gamma'$ is a loop with $1$ as the multiplicity of the unique vertex. Then the Brauer graph algebras $A_{\Gamma}$ and $A_{\Gamma'}$ are not stably equivalent.
\end{Lem}

\begin{proof}
We have $A_{\Gamma}\cong k[X,Y]/(X^2-Y^2,XY)$ and $A_{\Gamma'}\cong k[X,Y]/(X^2,Y^2)$. Each indecomposable $A_{\Gamma}$-module (resp. $A_{\Gamma'}$-module) at the mouth of a tube of rank $1$ of the stable Auslander-Reiten quiver of $A_{\Gamma}$ (resp. $A_{\Gamma'}$) is of the form
$$\xymatrix@R=1pc@C=2pc
{&1\ar@{-}@/_/[dd]_{X=\lambda}\ar@{-}@/^/[dd]^{Y=\mu} & \\
M_{[\lambda:\mu]}:&& \\
&1 & ,}$$
where $[\lambda:\mu]\in\mathbb{P}_{k}^{1}$. A calculation shows that $\Omega_{A_{\Gamma}}(M_{[\lambda:\mu]})\cong M_{[\mu:\lambda]}$ and $\Omega_{A_{\Gamma'}}(M_{[\lambda:\mu]})\cong M_{[\lambda:\mu]}$. If there exists an equivalence $F:A_{\Gamma}\text{-}{\underline{\mathrm{mod}}}\rightarrow A_{\Gamma'}\text{-}{\underline{\mathrm{mod}}}$, then
$$F(M_{[0:1]})\cong F\Omega_{A_{\Gamma}}(M_{[1:0]})\cong\Omega_{A_{\Gamma'}}F(M_{[1:0]})\cong
F(M_{[1:0]})$$
and $M_{[0:1]}\cong M_{[1:0]}$ in $A_{\Gamma}\text{-}{\underline{\mathrm{mod}}}$, a contradiction. It follows that $A_{\Gamma}$ and $A_{\Gamma'}$ are not stably equivalent.
\end{proof}
	
	Since we always assume that a Brauer graph $\Gamma$ is connected, its associated Brauer graph algebra $A_{\Gamma}$ is indecomposable. Moreover, we have the following result. 

\begin{Thm}\textnormal{(cf. \cite[Theorem 1.1]{Sch})}\label{ssbBGA}
		An algebra $A$ is symmetric special biserial if and only if it is a Brauer graph algebra.
	\end{Thm}

Note that symmetric special biserial algebras form a subclass of symmetric stably biserial algebras. For the definitions of (symmetric) stably biserial algebra and (symmetric) special biserial algebra, see Appendix \ref{sec-appendix}.

\subsubsection{Faces}
\

 We now recall the definition of faces in a Brauer graph from \cite{OZ}, which correspond to Green walks \cite{GreenJA} (or $G$-cycles \cite{A}).

\begin{Def}
Let $\Gamma$ be a ribbon graph. A face of perimeter $m$ of $\Gamma$ is an equivalence class of primitive cyclic sequences $F=(h_1,h_2,\cdots,h_m)$, where $h_1,\cdots,h_m\in H(\Gamma)$ such that $h_{i+1}=\iota\rho(h_{i})$ for all $i$. We call a sequence $F$ primitive, if $F$ is not a power of another sequence of this form. Two sequences of such kind are considered equivalent if they agree after a cyclic permutation of their entries.
\end{Def}

We give following example to help readers understand this concept.

\begin{example}\label{exa:2-ribbon graph}
	Let $\Gamma_1$ and $\Gamma_2$ be the Brauer graphs with multiplicity one and the ribbon graphs are given as follows (the orientation around each vertex is clockwise), respectively.
	\begin{center}
\tikzset{every picture/.style={line width=0.75pt}} 

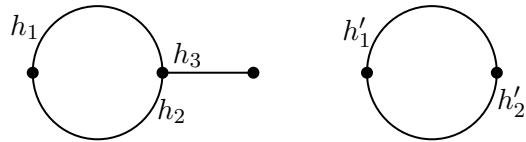
\begin{tikzpicture}[x=0.75pt,y=0.75pt,yscale=-1,xscale=1]

\draw   (61.1,94.92) .. controls (61.1,76.46) and (75.68,61.5) .. (93.67,61.5) .. controls (111.67,61.5) and (126.25,76.46) .. (126.25,94.92) .. controls (126.25,113.37) and (111.67,128.33) .. (93.67,128.33) .. controls (75.68,128.33) and (61.1,113.37) .. (61.1,94.92) -- cycle ;
\draw   (228.86,94.92) .. controls (228.86,76.46) and (243.45,61.5) .. (261.44,61.5) .. controls (279.43,61.5) and (294.02,76.46) .. (294.02,94.92) .. controls (294.02,113.37) and (279.43,128.33) .. (261.44,128.33) .. controls (243.45,128.33) and (228.86,113.37) .. (228.86,94.92) -- cycle ;
\draw    (126.25,94.92) -- (171.86,94.92) ;
\draw  [fill={rgb, 255:red, 0; green, 0; blue, 0 }  ,fill opacity=1 ] (169.36,94.92) .. controls (169.36,93.54) and (170.48,92.42) .. (171.86,92.42) .. controls (173.24,92.42) and (174.36,93.54) .. (174.36,94.92) .. controls (174.36,96.3) and (173.24,97.42) .. (171.86,97.42) .. controls (170.48,97.42) and (169.36,96.3) .. (169.36,94.92) -- cycle ;
\draw  [fill={rgb, 255:red, 0; green, 0; blue, 0 }  ,fill opacity=1 ] (123.75,94.92) .. controls (123.75,93.54) and (124.87,92.42) .. (126.25,92.42) .. controls (127.63,92.42) and (128.75,93.54) .. (128.75,94.92) .. controls (128.75,96.3) and (127.63,97.42) .. (126.25,97.42) .. controls (124.87,97.42) and (123.75,96.3) .. (123.75,94.92) -- cycle ;
\draw  [fill={rgb, 255:red, 0; green, 0; blue, 0 }  ,fill opacity=1 ] (58.6,94.92) .. controls (58.6,93.54) and (59.72,92.42) .. (61.1,92.42) .. controls (62.48,92.42) and (63.6,93.54) .. (63.6,94.92) .. controls (63.6,96.3) and (62.48,97.42) .. (61.1,97.42) .. controls (59.72,97.42) and (58.6,96.3) .. (58.6,94.92) -- cycle ;
\draw  [fill={rgb, 255:red, 0; green, 0; blue, 0 }  ,fill opacity=1 ] (226.36,94.92) .. controls (226.36,93.54) and (227.48,92.42) .. (228.86,92.42) .. controls (230.24,92.42) and (231.36,93.54) .. (231.36,94.92) .. controls (231.36,96.3) and (230.24,97.42) .. (228.86,97.42) .. controls (227.48,97.42) and (226.36,96.3) .. (226.36,94.92) -- cycle ;
\draw  [fill={rgb, 255:red, 0; green, 0; blue, 0 }  ,fill opacity=1 ] (291.52,94.92) .. controls (291.52,93.54) and (292.63,92.42) .. (294.02,92.42) .. controls (295.4,92.42) and (296.52,93.54) .. (296.52,94.92) .. controls (296.52,96.3) and (295.4,97.42) .. (294.02,97.42) .. controls (292.63,97.42) and (291.52,96.3) .. (291.52,94.92) -- cycle ;

\draw (48.34,64.2) node [anchor=north west][inner sep=0.75pt]   [align=left] {$h_1$};
\draw (130,78.23) node [anchor=north west][inner sep=0.75pt]   [align=left] {$h_3$};
\draw (122,106.3) node [anchor=north west][inner sep=0.75pt]   [align=left] {$h_2$};
\draw (215,65.53) node [anchor=north west][inner sep=0.75pt]   [align=left] {$h_1'$};
\draw (292.66,100.29) node [anchor=north west][inner sep=0.75pt]   [align=left] {$h_2'$};
\end{tikzpicture}
	\end{center}

The graph $\Gamma_1$ contains two faces: one of perimeter $2$, $\{h_1,h_2\}$, and another one of perimeter $4$, $\{\iota(h_1),\iota(h_3),h_3,\iota(h_2)\}$.

In contrast, $\Gamma_2$ has two faces, both of perimeter $2$: $\{h_1',h_2'\}$ and $\{\iota(h_1'),\iota(h_2')\}$.
\end{example}

We also note that if $\Gamma$ has a face of perimeter $1$, then $\Gamma$ must contain a loop.

\subsection{Stable equivalence of Morita type and its invariants}\label{subsec:stb-and-center}
\

In this section we recall stable equivalence of Morita type between two finite-dimensional $k$-algebras. For further details we refer the reader to~\cite[Chapter 5]{Z}.

\begin{Def}
	Let $A$ and $B$ be two finite-dimensional $k$-algebras. $A$ and $B$ are said to be stably equivalent of Morita type if there exist bimodules
\[
{}_A M_B \quad \text{and} \quad {}_B N_A
\]
such that:
\begin{enumerate}
  \item $M$ is projective as a left $A$-module and as a right $B$-module;
  \item $N$ is projective as a left $B$-module and as a right $A$-module;
  \item There exist projective bimodules $P$ (an $A$–$A$-bimodule) and $Q$ (a $B$–$B$-bimodule) such that
  \[
  M \otimes_B N \cong A \oplus P \quad \text{as } A\text{–}A\text{-bimodules,}
  \]
  and
  \[
  N \otimes_A M \cong B \oplus Q \quad \text{as } B\text{–}B\text{-bimodules.}
  \]
\end{enumerate}

\noindent
In this situation, the tensor functor
\[
M \otimes_B - \; : \; B\text{-}\underline{\mathrm{mod}}\longrightarrow A\text{-}\underline{\mathrm{mod}}
\]
induces an equivalence between the stable module categories of left modules,
with quasi-inverse
\[
N \otimes_A - \; : \; A\text{-}\underline{\mathrm{mod}} \longrightarrow B\text{-}\underline{\mathrm{mod}}.
\]

Moreover, an equivalence functor $\phi:A\text{-}\underline{\mathrm{mod}} \longrightarrow B\text{-}\underline{\mathrm{mod}}$ is said to be a stable equivalence of Morita type if it is natural isomorphic to a functor
\[
N \otimes_A - \; : \; A\text{-}\underline{\mathrm{mod}} \longrightarrow B\text{-}\underline{\mathrm{mod}},
\]
where $N$ is an $B$-$A$-bimodule which satisfies conditions as above.
\end{Def}

Naturally, a stable equivalence of Morita type induces an equivalence between the stable categories of two self-injective algebras as triangulated categories, and it preserves the symmetry of algebras (see for example, in~\cite[Proposition 5.5.4]{Z}).

Now we recall some basic notions concerning the centers of finite-dimensional $k$-algebras. Let $A$ be a finite-dimensional $k$-algebra.
The {\it center} of $A$ is
  \[
  Z(A) = \{\, z \in A \mid za = az \text{ for all } a \in A \,\}.
  \]
  It is well-known that $Z(A) \cong \operatorname{End}_{A^e}(A)$, where $A^e=A\otimes_{k}A^{op}$ is the enveloping algebra of $A$.

Recall that for a $k$-algebra $A$, the {\it Reynolds ideal} $R(A)$ is defined by $R(A)=Z(A)\cap\soc(A)$.  For a Brauer graph algebra $A$, we have $\soc(A)\subseteq Z(A)$ and $R(A)=\mathrm{soc}(A)$.

\begin{proposition}\textnormal{(\cite[Theorem 1.7]{ZZ})}\label{prop:ZZ}
	Let $A$ and $B$ be two symmetric indecomposable $k$-algebras which are stably equivalent of Morita type. If $k$ is of positive characteristic, then we have an isomorphism of algebras $Z(A)/R(A)\cong Z(B)/R(B)$.
\end{proposition}

Recall that for a $k$-algebra $A$, the {\it projective center} $Z^{\mathrm{pr}}(A)$ of $A$ is the ideal in $Z(A)$ consisting of $A^e$-homomorphisms which factor through projective $A^e$-modules. The {\it stable center} $Z^{\mathrm{st}}(A)$ of $A$ is the quotient algebra $Z(A)/Z^{\mathrm{pr}}(A)$. These objects fit into the following commutative diagram with exact rows:
$$\begin{tikzcd}[column sep=large]
0 \arrow[r]
& Z^{\mathrm{pr}}(A)
  \arrow[r, hook]
  \arrow[d, "\cong"]
& Z(A)
  \arrow[r]
  \arrow[d, "\cong"]
& Z^{\mathrm{st}}(A)
  \arrow[r]
  \arrow[d, "\cong"]
& 0 \\
0 \arrow[r]
& K
  \arrow[r, hook]
& \operatorname{End}_{A^e}(A)
  \arrow[r]
& \underline{\operatorname{End}}_{A^e}(A)
  \arrow[r]
& 0 ~.
\end{tikzcd}$$

\begin{proposition}\textnormal{(see \cite{B1})}\label{prop:P}
	Let $A$ and $B$ be two self-injective $k$-algebras which are stably equivalent of Morita type. Then we have an isomorphism of algebras $Z^{\mathrm{st}}(A)\cong Z^{\mathrm{st}}(B)$.
\end{proposition}

\subsection{Centers of Brauer graph algebras}\label{sec:center-of-SSB}
\

For Brauer graph algebras we now recall the results on their centers from \cite{AZ2}.

Assume that $A$ is a Brauer graph algebra defined by a Brauer graph $(\Gamma, m)$. Let $\{C_1, C_2, \ldots, C_r\}$ be the set of cycles of the form $C_\alpha$ in $Q_\Gamma$ (see Subsection~\ref{subsec:def-BGA} for the definition of $C_\alpha$), considered up to cyclic permutation, where each $C_i$ corresponds to a vertex $x_i$ of $\Gamma$, and denote by $m_i=m(C_i)$. For each $1 \le i \le r$, consider a cyclic sequence $(\alpha_{i,1}, \alpha_{i,2}, \ldots, \alpha_{i,l_i})$ of arrows in the cycle $C_i$, where $\sigma(\alpha_{i,j}) = \alpha_{i,j+1}$ and $l_i$ denotes the length of $C_i$. Let $r' \leq r$ be an integer such that $m_i > 1$ for $i = 1, \ldots, r'$ and $m_i = 1$ for $i = r' + 1, \ldots, r$.

For each loop $\gamma = \alpha_{i,j}$ such that $\sigma(\gamma) \ne \gamma$, set
\[
q_{\gamma} = q_{\alpha_{i,j}}
= \alpha_{i,j+1} \cdots \alpha_{i,l_i} \alpha_{i,1} \cdots \alpha_{i,j-1}
(\alpha_{i,j} \alpha_{i,j+1} \cdots \alpha_{i,l_i} \alpha_{i,1} \cdots \alpha_{i,j-1})^{m_i-1}.
\]
For each vertex $v$ of $Q$, let $\alpha_{i,j}$ be an arrow ending at $v$, and set
\[
s_v = (\alpha_{i,j} \cdots \alpha_{i,l_i} \alpha_{i,1} \cdots \alpha_{i,j-1})^{m_i}.
\]
Note that $s_v$ does not depend on $\alpha_{i,j}$ according to the defining relation $(a)$ in $I_{\Gamma}$, and $s_v$ is a socle element of $A$ corresponding to the vertex $v$. The following result is proved for the broader class of symmetric stably biserial algebras in \cite{AZ2}.

\begin{Prop}\textnormal{(see \cite[Proposition 4.1]{AZ2})}\label{prop:center-of-stBA}
Let $A=kQ/I$ be a Brauer graph algebra defined by a Brauer graph $(\Gamma, m)$ as above. As a vector space over $k$, the center $Z(A)$ is generated by $1$ together with the following elements:
\begin{itemize}
  \item[$(a)$] Elements
  \[
  p_{i,t} =
  (\alpha_{i,1}\alpha_{i,2}\cdots\alpha_{i,l_i})^t
  + (\alpha_{i,2}\cdots\alpha_{i,l_i}\alpha_{i,1})^t
  + \cdots
  + (\alpha_{i,l_i}\alpha_{i,1}\cdots\alpha_{i,l_i-1})^t,
  \]
  for $i = 1, 2, \ldots, r'$ and $t = 1, \ldots, m_i - 1$.

  \item[$(b)$] Elements $q_{\gamma}$ for each loop $\gamma$ such that $\sigma(\gamma) \ne \gamma$.

  \item[$(c)$] Elements $s_v$ for each vertex $v \in Q_0$.
\end{itemize}
\end{Prop}

Therefore the center $Z(A)$ only depends on the associated Brauer graph $(\Gamma, m)$ of $A$ and is free of the characteristic of the field $k$. We also note that the description of $Z(A)/\operatorname{soc}(Z(A))$ given in \cite[Proposition~4.1]{AZ2} requires a correction, since some exceptional cases are missing. We give the corrected statement below.

\begin{Prop}\label{remark-center-of-stBA}
Let $A=A_{\Gamma}$ be the Brauer graph algebra associated with a Brauer
graph $(\Gamma,m)$, with the notation introduced above.

\begin{enumerate}
\item Suppose that $\Gamma$ consists of a single loop 
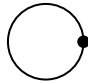
\[
\begin{tikzpicture}[baseline=-0.5ex]
\draw (0,0) circle (0.5);
\fill (0.5,0) circle (0.5ex);
\end{tikzpicture}
\] and that its unique vertex has multiplicity one. Then $Z(A)\cong A \cong k[X,Y]/(X^2,Y^2)$, $\operatorname{soc}(Z(A))=kXY$, and hence
$$Z(A)/\operatorname{soc}(Z(A))\cong k[X,Y]/(X^2,XY,Y^2).$$

\item Suppose that $\operatorname{char}(k)=2$ and that $\Gamma$ has a
unique vertex, whose multiplicity is $m\geq 2$. Then as a vector space over $k$, $\operatorname{soc}(Z(A))$ is spanned by the elements of types $(b)$ and $(c)$ in Proposition \ref{prop:center-of-stBA} together with $p_{1,m-1}$, and
\[Z(A)/\operatorname{soc}(Z(A))\cong k[X]/(X^{m-1}).\]

\item In all other cases, as a vector space over $k$, $\operatorname{soc}(Z(A))$ is spanned by the elements of types $(b)$ and $(c)$ in Proposition \ref{prop:center-of-stBA}, and
\[Z(A)/\operatorname{soc}(Z(A))\cong k[X_1,\ldots,X_{r'}]\big/
(
X_i^{m_i},\,X_iX_j\mid i\neq j
) .\]
\end{enumerate}
\end{Prop}

\begin{proof}
The first assertion follows from a direct computation.

For $1\leq i\leq r'$ and $s,t\geq 1$, multiplication of the central elements of type $(a)$ gives $p_{i,s}p_{j,t}=0$, if $i\neq j$, and
\[
p_{i,s}p_{i,t}
=
\begin{cases}
p_{i,s+t}, & s+t<m_i,\\[1mm]
\displaystyle\sum_{h\in H_{x_i}}s_{\bar h},
    & s+t=m_i,\\[3mm]
0, & s+t>m_i,
\end{cases}
\]
where $x_i$ is the vertex of $\Gamma$ corresponding to $C_i$, and $s_{\bar h}$ is the socle element associated with the edge $\bar h$. Suppose first that $\Gamma$ has more than one vertex. Since $\Gamma$ is connected, every vertex $x_i$ is incident with at least one non-loop edge. The corresponding socle element occurs exactly once in
$$p_{i,m_i-1}p_{i,1}=\sum_{h\in H_{x_i}}s_{\bar h},$$
so this product is nonzero. The same conclusion holds when $\Gamma$ has
a unique vertex and $\operatorname{char}(k)\neq 2$.

Now suppose that $\Gamma$ has a unique vertex. Every edge of $\Gamma$ is then a loop, and each socle element occurs twice in the preceding sum. Consequently,
$$p_{1,m-1}p_{1,1}=2\sum_{e\in E(\Gamma)}s_e.$$
If $\operatorname{char}(k)=2$, this product is zero. Moreover, $p_{1,m-1}$ annihilates all the remaining noninvertible central elements. Thus
$p_{1,m-1}\in\operatorname{soc}(Z(A))$, and the images of
$1,p_{1,1},\ldots,p_{1,m-2}$ form a basis of $Z(A)/\operatorname{soc}(Z(A))$. This proves $$Z(A)/\operatorname{soc}(Z(A))\cong k[X]/(X^{m-1})$$
and the second assertion is verified.

Except when $\Gamma$ consists of a single loop of multiplicity one, direct multiplication shows that the elements of types $(b)$ and $(c)$ annihilate the radical of $Z(A)$. The multiplication formulas above also show that no other element of type $(a)$ belongs to the socle, apart from $p_{1,m-1}$ in the exceptional characteristic $2$ case. This proves the third assertion.
\end{proof}

\begin{Ex}\label{ex:center-quotient-does-not-detect-multiplicity}
Assume that $\operatorname{char}(k)=2$. Let $\Gamma$ be the Brauer graph consisting of one vertex and two loops:
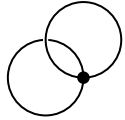
\[
\begin{tikzpicture}[baseline=-0.5ex]
\draw (-0.5,0) circle (0.5);
\draw[
  preaction={draw=white,line width=2pt}
] (0,0.5) circle (0.5);
\fill (0,0) circle (0.5ex);
\end{tikzpicture}
\]
Choose the cyclic ordering of the four half-edges to be $h_1,\ h_2,\ \iota(h_1),\ \iota(h_2)$. Let $C_1,C_2,C_3,C_4$ denote the four cyclic rotations of the corresponding cycle.

Let $A_1$ be the Brauer graph algebra for which the unique vertex of
$\Gamma$ has multiplicity one. Thus $C_1=C_3$, $C_2=C_4$. There are no central elements of type $(b)$ for the chosen cyclic ordering, and
\[Z(A_1)
=\operatorname{span}_k\{1,C_1,C_2\},
\qquad
C_iC_j=0\quad (1\leq i,j\leq 2).\]
Therefore
\[\operatorname{soc}(Z(A_1))=\operatorname{span}_k\{C_1,C_2\},
\qquad
Z(A_1)/\operatorname{soc}(Z(A_1))\cong k.\]

Let $A_2$ be the Brauer graph algebra obtained from the same Brauer graph
by assigning multiplicity two to its unique vertex. Set $s_1=C_1^2=C_3^2$,
$s_2=C_2^2=C_4^2$, $p=C_1+C_2+C_3+C_4$. Then
\[Z(A_2)=\operatorname{span}_k\{1,p,s_1,s_2\}.\]
Since $\operatorname{char}(k)=2$, we have
\[p^2=C_1^2+C_2^2+C_3^2+C_4^2=2s_1+2s_2=0.\]
Moreover, $ps_i=s_is_j=0$ ($1\leq i,j\leq 2$). It follows that
\[\operatorname{soc}(Z(A_2))=\operatorname{span}_k\{p,s_1,s_2\},
\qquad Z(A_2)/\operatorname{soc}(Z(A_2))\cong k.\]

Therefore, $Z(A_1)/\operatorname{soc}(Z(A_1))\cong Z(A_2)/\operatorname{soc}(Z(A_2))$ although the corresponding multi-sets of vertex multiplicities are $\{1\}$ and $\{2\}$.
\end{Ex}

\begin{Rem}\label{rem:correction-to-AZ-multiplicities}
The algebras $A_1$ and $A_2$ in Example~\ref{ex:center-quotient-does-not-detect-multiplicity} are not derived equivalent since $$\mathrm{dim}_k Z(A_1)=3\neq4=\mathrm{dim}_k Z(A_2).$$
However, their quotients by the socles of their centers are isomorphic.

Therefore, the algebra $Z(A)/\operatorname{soc}(Z(A))$ does not, in general, determine the multi-set of vertex multiplicities. Consequently, the assertion in \cite[Proposition~4.5]{AZ2} concerning the derived invariance of vertex multiplicities cannot be deduced directly from the description in \cite[Proposition~4.1]{AZ2}.
We note that the assertion itself is correct. Proposition~\ref{thm-multi-sets-of-multiplicity} gives an alternative proof and establishes the stronger statement that the multi-set of vertex multiplicities is preserved under stable equivalences of Morita type.
\end{Rem}

\begin{Cor}\label{cor-center-quotient}
Let $A=kQ/I$ be a Brauer graph algebra defined by a Brauer graph $(\Gamma, m)$ in the beginning of Subsection \ref{sec:center-of-SSB}, and let $R(A)=Z(A)\cap\soc(A)$ be the Reynolds ideal of $A$. Then as a $k$-algebra,
$$Z(A)/R(A)\cong k[X_1, X_2, \ldots, X_{r'}, Y_1,\cdots,Y_l]/ J,$$
where $J$ is the ideal in $k[X_1, X_2, \ldots, X_{r'}, Y_1,\cdots,Y_l]$ generated by $X_i^{m_i}$, $X_i X_j$ $(i,j=1,2,\cdots,r'$, $i \ne j)$, $Y_{p}Y_{q}$, $Y_p X_i$ $(p,q=1,2,\cdots,l$, $i=1,2,\cdots,r')$,
and $l$ is the number of loops $\gamma$ in $Q$ such that $\sigma(\gamma) \ne \gamma$ (or equivalently, $l$ is the number of faces of perimeter $1$ of $\Gamma$).
\end{Cor}

\begin{proof}
Since $\soc(A)\subseteq Z(A)$, $R(A)=\soc(A)$ is the subspace of $A$ generated by elements $s_v$ for each vertex $v \in Q_0$. Therefore as a vector space over $k$, $Z(A)/R(A)$ is generated by $1$ together with elements of types $(a)$ and $(b)$ in Proposition \ref{prop:center-of-stBA}. For elements of type $(a)$, we have $p_{i,1}^t=p_{i,t}$ and $p_{i,1}^{m_i}\in\soc(A)$ for $i = 1, 2, \cdots, r'$ and $t = 1, \ldots, m_i - 1$, and $p_{i,1}p_{j,1}=0$ if $i\neq j$. For each two elements $q_{\beta}$, $q_{\gamma}$ of type $(b)$, we have $q_{\beta}q_{\gamma}\in\soc(A)$ and $q_{\beta}p_{i,t}=0$ for $i = 1, 2, \cdots, r'$ and $t = 1, \ldots, m_i - 1$. Denote by
$$\pi:k[X_1, X_2, \ldots, X_{r'}, Y_1,\cdots,Y_l]\rightarrow Z(A)/R(A)$$
the $k$-algebra homomorphism given by $\pi(X_i)=p_{i,1}$ and $\pi(Y_j)=q_{\gamma_j}$, where $\gamma_j$ $(j=1,2,\cdots,l)$ are all loops in $Q$ with $\sigma(\gamma_j)\neq\gamma_j$. Since the elements $p_{i,1}$ and $q_{\gamma_j}$ in $Z(A)/R(A)$ generate $Z(A)/R(A)$ as a $k$-algebra, which satisfy the relations in $J$, $\pi$ induces a surjective $k$-algebra homomorphism $k[X_1, X_2, \ldots, X_{r'}, Y_1,\cdots,Y_l]/ J\rightarrow Z(A)/R(A)$. By comparing dimensions we conclude that $Z(A)/R(A)\cong k[X_1, X_2, \ldots, X_{r'}, Y_1,\cdots,Y_l]/ J$.
\end{proof}

\section{Classification of Brauer graph algebras under stable equivalence of Morita type}\label{sec:classification-BGA-stM}

\subsection{Cartan matrix and stable center}
\

In this subsection, let $\Gamma$ be a Brauer graph with $r$ vertices $x_1,\cdots,x_r$ and $n$ edges $e_1,\cdots,e_n$, where the multiplicity of the vertex $x_i$ is $m_i$. Suppose that there is an integer $1\leq r'\leq r$ such that $m_i > 1$ for $i = 1,\cdots,r'$ and $m_i=1$ for $i=r'+1,\cdots,r$. Let $C_i$ be a cycle in $Q_{\Gamma}$ of the form $C_{\alpha}$ corresponding to the vertex $x_i$ (see Subsection~\ref{subsec:def-BGA} for the definition of $C_\alpha$).  Let $v_i$ be the vertex of $Q_{\Gamma}$ which corresponds to the edge $e_i$ of $\Gamma$. Denote by $s_i=s_{v_i}$ the socle element of $A=A_{\Gamma}$ corresponding to the vertex $v_i$ (see Subsection \ref{sec:center-of-SSB}), and denote by $p_{i,t}$ the element of $Z(A)$ in Proposition \ref{prop:center-of-stBA}.

Define a matrix $D=[d_{ij}]_{r\times n}$ as follows: for each $1\leq i\leq r$ and each $1\leq j\leq n$, $d_{ij}$ is the multiplicity of incidence of the edge $e_j$ with the vertex $x_i$. (If $e_j$ is a loop at $x_i$, define $d_{ij}=2$; if $e_j$ is not incident at $x_i$, define $d_{ij}=0$; otherwise define $d_{ij}=1$.) Denote by $D'$ the transpose of the matrix $D$.

\begin{Lem}\textnormal{(cf. \cite[Section 5]{C})}\label{lem-Cartan-matrix}
The Cartan matrix $C(A)$ of $A=A_{\Gamma}$ is $D'\mathrm{diag}(m_1,\cdots,m_r)D$.
\end{Lem}

\begin{proof}
Let $P_j$ (resp. $S_j$) be the indecomposable projective $A$-module (resp. simple $A$-module) corresponding to the edge $e_j$ of $\Gamma$. Then $C(A)=[c_{j_1 j_2}]_{n\times n}$, where $c_{j_1 j_2}$ is the number of times that $S_{j_2}$ appears as a composition factor of $P_{j_1}$. It suffices to show that $c_{j_1 j_2}=\sum_{i=1}^{r}m_i d_{i j_1}d_{i j_2}$ for all $j_1,j_2\in\{1,2,\cdots n\}$.

Since each edge $e_j$ in $\Gamma$ is incident with at most two vertices, there exists at most two numbers $i,i'\in\{1,2,\cdots,r\}$ such that $d_{ij}\neq 0$ and $d_{i' j}\neq 0$. If there exist two different numbers $i_1,i_2\in\{1,2,\cdots,r\}$ such that $d_{i_1 j_1}d_{i_1 j_2}\neq 0$ and $d_{i_2 j_1}d_{i_2 j_2}\neq 0$, then both $e_{j_1}$ and $e_{j_2}$ are edges that connect the vertices $x_{i_1}$ and $x_{i_2}$, and $d_{i_1 j_1}d_{i_1 j_2}=d_{i_2 j_1}d_{i_2 j_2}=1$. Therefore $$c_{j_1 j_2}=m_{i_1}+m_{i_2}=\sum_{i=1}^{r}m_i d_{i j_1}d_{i j_2}.$$

If there exists exactly one number $i_1\in\{1,2,\cdots,r\}$ such that $d_{i_1 j_1}d_{i_1 j_2}\neq 0$, then $x_{i_1}$ is the only vertex that both $e_{j_1}$ and $e_{j_2}$ are incident with. There are three possible cases: $(1)$ Both $e_{j_1}$ and $e_{j_2}$ are loops at $x_{i_1}$; $(2)$ One of the edges between $e_{j_1}$ and $e_{j_2}$ is a loop at $x_{i_1}$, and the other is an ordinary edge that connect $x_{i_1}$ and some other vertex $x_{i_2}$; $(3)$ Both $e_{j_1}$ and $e_{j_2}$ are ordinary edges that connect $x_{i_1}$ and some other vertices. In case $(1)$ we have $d_{i_1 j_1}d_{i_1 j_2}=4$ and $d_{i j_1}d_{i j_2}=0$ for each $i\neq i_1$. Then $$\sum_{i=1}^{r}m_i d_{i j_1}d_{i j_2}=4m_{i_1}=c_{j_1 j_2}.$$
In case $(2)$ we have $d_{i_1 j_1}d_{i_1 j_2}=2$ and $d_{i j_1}d_{i j_2}=0$ for each $i\neq i_1$. So $\sum_{i=1}^{r}m_i d_{i j_1}d_{i j_2}=2m_{i_1}=c_{j_1 j_2}$. In case $(3)$ we have $d_{i_1 j_1}d_{i_1 j_2}=1$ and $d_{i j_1}d_{i j_2}=0$ for each $i\neq i_1$. Then $$\sum_{i=1}^{r}m_i d_{i j_1}d_{i j_2}=m_{i_1}=c_{j_1 j_2}.$$

If $d_{i j_1}d_{i j_2}=0$ for all $i\in\{1,2,\cdots,r\}$, then there exists no vertex in $\Gamma$ that both $e_{j_1}$ and $e_{j_2}$ are incident with. Thus $c_{j_1 j_2}=0=\sum_{i=1}^{r}m_i d_{i j_1}d_{i j_2}$.
\end{proof}

For the computation of the projective center of $A$, we first recall the definition of the Higman ideal for symmetric algebras. Let $k$ be an algebraically closed field and $\Lambda$ a finite-dimensional symmetric $k$-algebra. Let $\{a_1, \ldots, a_m\}$ and $\{b_1, \ldots, b_m\}$ be a pair of dual bases of $\Lambda$. The trace map $\mathrm{Tr}\colon \Lambda \to \Lambda$ is defined by the $k$-linear map
\[
\mathrm{Tr}(x) = \sum_{i=1}^m b_i x a_i.
\]
By \cite[Lemma 4.1]{HHKM}, the map $\mathrm{Tr}$ is independent of the choice of dual bases, and its image is contained in the Reynolds ideal $R(\Lambda)$ of $\Lambda$. We define the {\it Higman ideal} of $\Lambda$ as $$H(\Lambda) := \mathrm{Im}(\mathrm{Tr}).$$ Furthermore, by \cite[Proposition 3.13]{B2}, for symmetric algebras, the projective center $Z^{\mathrm{pr}}(\Lambda)$ coincides with the Higman ideal $H(\Lambda)$.

Moreover, suppose that $\Lambda=kR/L$ is a symmetric algebra given by a quiver $R$ with an admissible ideal $L$ in $kR$, where $R$ contains $n$ vertices $x_1,x_2,\cdots,x_n$, and denote by $e_i$ the primitive idempotent of $\Lambda$ corresponding to the vertex $x_i$ for $i=1,2,\cdots,n$. We choose a $k$-basis $\{a_1, \ldots, a_m\}$ of $\Lambda$ such that $a_i=e_i$ for $i=1,2,\cdots,n$ and $a_{n+1},\cdots,a_m$ form a $k$-basis of $\mathrm{rad}(\Lambda)$. Denote the dual basis of $\{a_1, \ldots, a_m\}$ by $\{b_1, \ldots, b_m\}$. We have that $b_1,\cdots,b_n$ form a $k$-basis of $\mathrm{soc}(\Lambda)$. Then according to \cite[Lemma 4.3]{HHKM}, we have
$$
\mathrm{Tr}(e_i) = \sum_{j=1}^n (\mathrm{dim}_k e_i \Lambda e_j)\, b_j,
$$
for $1\leq i\leq n$. Moreover, by \cite[Lemma 4.1]{HHKM} we have $\mathrm{Tr}(a_i) = 0$ for $n + 1 \leq i \leq m$. Therefore, $Z^{\mathrm{pr}}(\Lambda)=H(\Lambda)$ is the image of the linear transformation $f \colon \mathrm{soc}(\Lambda) \to \mathrm{soc}(\Lambda)$, whose matrix with respect to the basis $\{b_1, \cdots, b_n\}$ of $\mathrm{soc}(\Lambda)$ is the Cartan matrix $C(\Lambda)$ of $\Lambda$. In particular, the dimension of $Z^{\mathrm{pr}}(\Lambda)$ is equal to the rank of $C(\Lambda)$ as a matrix over $k$.

We now compute the projective center $Z^{\mathrm{pr}}(A)$ of the Brauer graph algebra $A=A_{\Gamma}$. By the discussion as above, $Z^{\mathrm{pr}}(A)$ is the image of the linear transformation $f:\mathrm{soc}(A)\rightarrow \mathrm{soc}(A)$, where the matrix of $f$ corresponding to the basis $\{s_1,\cdots,s_n\}$ of $\mathrm{soc}(A)$ is the Cartan matrix $C(A)$ of $A$.

\begin{Lem}\label{lem}
We have $(p_{1,1}^{m_1},\cdots,p_{r,1}^{m_r})'=D(s_1,\cdots,s_n)'$.
\end{Lem}

\begin{proof}
For each $1 \le i \le r$, consider a cyclic sequence $(\alpha_{i,1}, \alpha_{i,2}, \ldots, \alpha_{i,l_i})$ of arrows in the cycle $C_i$, where $\sigma(\alpha_{i,j}) = \alpha_{i,j+1}$ and $l_i$ denotes the length of $C_i$. We have
\begin{multline*}
  p_{i,1}^{m_i} =
  (\alpha_{i,1}\alpha_{i,2}\cdots\alpha_{i,l_i}
  +\alpha_{i,2}\alpha_{i,3}\cdots\alpha_{i,1}
  + \cdots
  +\alpha_{i,l_i}\alpha_{i,1}\cdots\alpha_{i,l_i-1})^{m_i} \\
  =
  (\alpha_{i,1}\alpha_{i,2}\cdots\alpha_{i,l_i})^{m_i}
  + (\alpha_{i,2}\alpha_{i,3}\cdots\alpha_{i,1})^{m_i}
  + \cdots
  + (\alpha_{i,l_i}\alpha_{i,1}\cdots\alpha_{i,l_i-1})^{m_i}.
\end{multline*}
Suppose that $\alpha_{i,j}=\alpha_{h_{i,j}}$, where $\{h_{i,1},h_{i,2},\cdots,h_{i,l_i}\}=H_{x_i}(\Gamma)$ (recall that $H_{x_i}(\Gamma)$ is the set of all half-edges $h\in H(\Gamma)$ such that $s(h)=x_i$). Let $\overline{h_{i,j}}=e_{p_j}$ for $j=1,2,\cdots,l_i$. Then $(\alpha_{i,j+1}\cdots\alpha_{i, l_i}\alpha_{i,1}\cdots\alpha_{i,j-1}\alpha_{i,j})^{m_i}=s_{p_j}$ and $p_{i,1}^{m_i}=\sum_{j=1}^{l_i}s_{p_j}$. Note that for each $k\in\{1,2,\cdots,n\}$, the cardinality of the set $\{j\in\{1,2,\cdots,l_i\}\mid p_j=k\}=\{j\in\{1,2,\cdots,l_i\}\mid \overline{h_{i,j}}=e_k\}$ is $d_{ik}$. Then $p_{i,1}^{m_i}=\sum_{k=1}^{n}d_{ik}s_k$, and $(p_{1,1}^{m_1},\cdots,p_{r,1}^{m_r})'=D(s_1,\cdots,s_n)'$.
\end{proof}

\begin{Lem}\label{lem-proj-center}
If $\mathrm{ckar}(k)=0$, then $Z^{\mathrm{pr}}(A)$ is the subspace of $Z(A)$ spanned by $p_{1,1}^{m_1},p_{2,1}^{m_2},\cdots,p_{r,1}^{m_r}$.
\end{Lem}

\begin{proof}
$Z^{\mathrm{pr}}(A)$ is the image of the linear transformation $f:\mathrm{soc}(A)\rightarrow \mathrm{soc}(A)$, where the matrix of $f$ corresponding to the basis $\{s_1,\cdots,s_n\}$ of $\mathrm{soc}(A)$ is $C(A)$. By Lemma \ref{lem-Cartan-matrix}, $C(A)=D'\mathrm{diag}(m_1,\cdots,m_r)D$. Then $$(f(s_1),f(s_2),\cdots,f(s_n))=(s_1,s_2,\cdots,s_n)D'\mathrm{diag}(m_1,\cdots,m_r)D.$$
By Lemma \ref{lem}, $(p_{1,1}^{m_1},\cdots,p_{r,1}^{m_r})=(s_1,s_2,\cdots,s_n)D'$. Then
$$(f(s_1),f(s_2),\cdots,f(s_n))=(p_{1,1}^{m_1},\cdots,p_{r,1}^{m_r})\mathrm{diag}(m_1,\cdots,m_r)D$$
and each $f(s_i)$ is a linear combination of $p_{1,1}^{m_1},\cdots,p_{r,1}^{m_r}$.

Consider two systems of linear equations $DX=0$ and $D'\mathrm{diag}(m_1,\cdots,m_r)DX=0$ over $\mathbb{Q}$. Since $m_1,\cdots,m_r>0$, these two linear systems are equivalent. Then $D$ and $D'\mathrm{diag}(m_1,\cdots,m_r)D$ have the same rank as matrices over $\mathbb{Q}$. Since $k$ is a field of characteristic $0$, it contains $\mathbb{Q}$ as a subfield, and $D$ and $D'\mathrm{diag}(m_1,\cdots,m_r)D$ have the same rank as matrices over $k$. Therefore the sets of vectors $\{f(s_1),f(s_2),\cdots,f(s_n)\}$ and $\{p_{1,1}^{m_1},\cdots,p_{r,1}^{m_r}\}$ have the same rank. Since each $f(s_i)$ is a linear combination of $p_{1,1}^{m_1},\cdots,p_{r,1}^{m_r}$, the sets of vectors $\{f(s_1),f(s_2),\cdots,f(s_n)\}$ and $\{p_{1,1}^{m_1},\cdots,p_{r,1}^{m_r}\}$ are equivalent. Since $Z^{\mathrm{pr}}(A)$ is the image of $f$, it is the subspace of $A$ spanned by $p_{1,1}^{m_1},\cdots,p_{r,1}^{m_r}$.
\end{proof}

Let $Z^{\mathrm{st}}(A)=Z(A)/Z^{\mathrm{pr}}(A)$ be the stable center of $A$.

\begin{Prop}\label{stable center of BGA}
Suppose that $\mathrm{ckar}(k)=0$. If $\Gamma$ is not multiplicity-free,
then considered as a $k$-algebra,
$$Z^{\mathrm{st}}(A)/\mathrm{soc}(Z^{\mathrm{st}}(A))\cong k[X_1, X_2,\cdots, X_{r'}]/(X_{i}^{m_i-1},X_i X_j (i\neq j)).$$
Otherwise,  $Z^{\mathrm{st}}(A)/\mathrm{soc}(Z^{\mathrm{st}}(A))\cong k$ or
$Z^{\mathrm{st}}(A)/\mathrm{soc}(Z^{\mathrm{st}}(A))=0.$
\end{Prop}

\begin{proof}
Recall from Proposition \ref{prop:center-of-stBA} that $Z(A)$ is generated by $1$ together with elements $p_{i,t}$ ($i=1,2,\cdots,r'$, $t=1,2,\cdots,m_i-1$), $q_{\gamma}$ ($\gamma$ is a loop with $\sigma(\gamma)\neq\gamma$), $s_i$ ($i=1,2,\cdots,n$) as a $k$-vector space. According to Lemma \ref{lem-proj-center}, $Z^{\mathrm{pr}}(A)$ is the subspace of $Z(A)$ spanned by $p_{1,1}^{m_1},p_{2,1}^{m_2},\cdots,p_{r,1}^{m_r}$, which is contained in the subspace of $Z(A)$ spanned by $s_1,s_2,\cdots,s_n$.

If $\Gamma$ is the multiplicity-free Brauer graph
$$
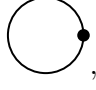
\begin{tikzpicture}
\draw (0,0) circle (0.5);
\fill (0.5,0) circle (0.5ex);
\end{tikzpicture},$$
then $A=Z(A)\cong k[X,Y]/(X^2,Y^2)$. By Lemma \ref{lem-proj-center}, $Z^{\mathrm{pr}}(A)=kXY$. Then $$Z^{\mathrm{st}}(A)\cong k[X,Y]/(X^2,XY,Y^2)$$
and
$$Z^{\mathrm{st}}(A)/\mathrm{soc}(Z^{\mathrm{st}}(A))\cong k.$$

Now suppose that $\Gamma$ is not the multiplicity-free Brauer graph
$$
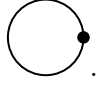
\begin{tikzpicture}
\draw (0,0) circle (0.5);
\fill (0.5,0) circle (0.5ex);
\end{tikzpicture}.$$
By Proposition \ref{remark-center-of-stBA}, $\operatorname{soc}(Z(A))$ is spanned by elements of the form $q_{\gamma}$, $s_i$. Denote by $\overline{a}$ the residue class of $a\in Z(A)$ in $Z^{\mathrm{st}}(A)=Z(A)/Z^{\mathrm{pr}}(A)$. Since each $q_{\gamma}$ and each $s_i$ belongs to $\operatorname{soc}(Z(A))$, $\overline{q_{\gamma}},\overline{s_i}\in\mathrm{soc}(Z^{\mathrm{st}}(A))$.

If $\Gamma$ is not multiplicity-free, then $r'\geq 1$. Since $0\neq\overline{p_{1,m_{1}-1}}\in\mathrm{rad}(Z^{\mathrm{st}}(A))$, $Z^{\mathrm{st}}(A)$ is non simple, and $\mathrm{soc}(Z^{\mathrm{st}}(A))$ is contained in $\mathrm{rad}(Z^{\mathrm{st}}(A))$. Consider the element $\overline{p_{i,m_{i}-1}}$ in $Z^{\mathrm{st}}(A)$ with $i=1,2,\cdots,r'$. Since each $q_{\gamma}$ and each $s_j$ belongs to $\operatorname{soc}(Z(A))$, we have $p_{i,m_{i}-1}q_{\gamma}=p_{i,m_{i}-1}s_j=0$, and $\overline{p_{i,m_{i}-1}}\overline{q_{\gamma}}=\overline{p_{i,m_{i}-1}}\overline{s_j}=0$ in $Z^{\mathrm{st}}(A)$. Moreover, since
\begin{equation*}
p_{i,m_{i}-1}p_{j,t}=\begin{cases}
p_{i,1}^{m_i}, & \text{ if } j=i \text{ and } t=1; \\
0, & \text{ otherwise},
\end{cases}
\end{equation*}
we have $\overline{p_{i,m_{i}-1}}\overline{p_{j,t}}=0$ in $Z^{\mathrm{st}}(A)$ for each $j=1,2,\cdots,r'$ and $t=1,2,\cdots,m_j-1$. Then $\overline{p_{i,m_{i}-1}}\in\mathrm{soc}(Z^{\mathrm{st}}(A))$. On the other hand, for each $\overline{a}\in\mathrm{soc}(Z^{\mathrm{st}}(A))$, suppose that $$a=\sum_{j=1}^{r'}\sum_{t=1}^{m_j-1}\lambda_{j,t}p_{j,t}+\sum_{\gamma}\lambda_{\gamma}q_{\gamma}+\sum_{l=1}^{n}\mu_{l}s_l,$$
where $\lambda_{j,t},\lambda_{\gamma},\mu_{l}\in k$. Then $0=\overline{a}\overline{p_{j,1}}=\sum_{t=1}^{m_j-2}\lambda_{j,t}\overline{p_{j,t+1}}$ in $Z^{\mathrm{st}}(A)$ for each $j=1,2,\cdots,r'$. Since $\overline{p_{j,2}},\cdots,\overline{p_{j,m_j-1}}$ are linearly independent in $Z^{\mathrm{st}}(A)$, we have $\lambda_{j,t}=0$ for each $j=1,2,\cdots,r'$ and $t=1,2,\cdots,m_j-2$. Then $$\overline{a}=\sum_{j=1}^{r'}\lambda_{j,m_j-1}\overline{p_{j,m_j-1}}+\sum_{\gamma}\lambda_{\gamma}\overline{q_{\gamma}}+\sum_{l=1}^{n}\mu_{l}\overline{s_l},$$
and $\mathrm{soc}(Z^{\mathrm{st}}(A))$ (as a $k$-vector space) is spanned by $\overline{p_{j,m_{j}-1}}$ ($j=1,2,\cdots,r'$) together with elements of the form $\overline{q_{\gamma}}$, $\overline{s_i}$. Denote by $I$ the ideal in $Z(A)$ spanned by $p_{j,m_{j}-1}$ ($j=1,2,\cdots,r'$) together with elements of the form $q_{\gamma}$, $s_i$ as a $k$-vector space. Then $\mathrm{soc}(Z^{\mathrm{st}}(A))=I/Z^{\mathrm{pr}}(A)$ and $Z^{\mathrm{st}}(A)/\mathrm{soc}(Z^{\mathrm{st}}(A))\cong Z(A)/I$. It is straightforward to show that $Z(A)/I$ is isomorphic to $k[X_1, X_2,\cdots, X_{r'}]/(X^{m_i-1},X^i X^j(i\neq j))$ as a $k$-algebra.

If $\Gamma$ is multiplicity-free, then $r'=0$ and $Z(A)$ is generated by $1$ together with elements $q_{\gamma}$ ($\gamma$ is a loop with $\sigma(\gamma)\neq\gamma$), $s_i$ ($i=1,2,\cdots,n$) as a $k$-vector space. By Proposition \ref{remark-center-of-stBA}, $\operatorname{soc}(Z(A))$ is spanned by elements of the form $q_{\gamma}$, $s_i$. Then $Z(A)/\operatorname{soc}(Z(A))\cong k$. Since $Z^{\mathrm{st}}(A)/\mathrm{soc}(Z^{\mathrm{st}}(A))$ is a quotient of $Z(A)/\operatorname{soc}(Z(A))$, it is isomorphic to $k$ or is equal to $0$.
\end{proof}

\subsection{Invariants of stable equivalence of Morita type between Brauer graph algebras}
\

Combining \cite[Propositions 2.12, 3.1 and Theorem 3.3]{A} and \cite[Theorem 2]{AZ1}, we obtain some invariants of Brauer graph algebras under stable equivalence.

\begin{Thm}\label{thm:ant's-stable-invariants}
Let $A$ and $A'$ be Brauer graph algebras over an algebraically closed field $k$, associated with Brauer graphs $\Gamma$ and $\Gamma'$ respectively. Moreover, suppose that $\Gamma$ (resp. $\Gamma'$) is neither a loop with $1$ as the multiplicity of the unique vertex nor an edge with $2$ as the multiplicity of both vertices. If the stable categories of $A$ and $A'$ are equivalent as triangulated categories, then all of the following conditions are satisfied:
 \begin{enumerate}
		\item $\Gamma$ and $\Gamma'$ have same number of vertices, edges and faces;
		\item $\Gamma$ and $\Gamma'$ have same multi-sets of perimeters of faces;
		\item either both or none of $\Gamma$ and $\Gamma'$ are bipartite.
 \end{enumerate}
\end{Thm}

The reason for us to add restrictions on $\Gamma$ and $\Gamma'$ in above Theorem is that we need to ensure that $\Gamma$ (resp. $\Gamma'$) is uniquely determined by $A$ (resp. $A'$) up to isomorphism, see Proposition \ref{prop-uniqueness-of-BG} and Example \ref{ex-isomorphic-BGAs}.

Using Theorem \ref{thm:ant's-stable-invariants}, \cite{AZ2} and \cite{OZ} provide a complete characterization of derived equivalences for Brauer graph algebras.

\begin{Thm}\textnormal{(\cite[Theorem 7.12]{OZ})}\label{thm:derived-eqv-between-BGAs}
Let $A$ and $A'$ be Brauer graph algebras over an algebraically closed field $k$, associated with Brauer graphs $\Gamma$ and $\Gamma'$ respectively, and suppose that $A$ and $A'$ are not local. Then the following statements are equivalent:
 \begin{enumerate}
	\item $A$ and $A'$ are derived equivalent;
	\item all of the following conditions are satisfied:
	\begin{itemize}
		\item $\Gamma$ and $\Gamma'$ have same number of vertices, edges and faces;
		\item $\Gamma$ and $\Gamma'$ have same multi-sets of perimeters of faces;
		\item $\Gamma$ and $\Gamma'$ have same multi-sets of multiplicities of vertices;
		\item either both or none of $\Gamma$ and $\Gamma'$ are bipartite.
	\end{itemize}
 \end{enumerate}
\end{Thm}

Thus, by comparing Theorem \ref{thm:ant's-stable-invariants} and Theorem \ref{thm:derived-eqv-between-BGAs}, the key issue in determining whether two Brauer graph algebras which are stably equivalent of Morita type are derived equivalent, lies in examining whether their corresponding Brauer graphs have the same multi-sets of vertex multiplicities. The result below gives an affirmative answer to this question.

\begin{Prop}\label{thm-multi-sets-of-multiplicity}
Let $\Gamma$ and $\Gamma'$ be Brauer graphs, which are neither a loop with $1$ as the multiplicity of the unique vertex nor an edge with $2$ as the multiplicity of both vertices. Let $A=A_{\Gamma}$ and $A'=A_{\Gamma'}$ be the corresponding Brauer graph algebras. If $A$ and $A'$ are stably equivalent of
Morita type, then $\Gamma$ and $\Gamma'$ have the same multi-sets of multiplicities of vertices.
\end{Prop}

\begin{proof}
Suppose that $\Gamma$ (resp. $\Gamma'$) has $r$ (resp. $r'$) vertices $x_1,\cdots,x_r$ (resp. $x'_1,\cdots,x'_{r'}$) and $n$ (resp. $n'$) edges, where the multiplicity of the vertex $x_i$ (resp. $x'_i$) is $m_i$ (resp. $m'_i$). Moreover, suppose that the number of faces of perimeter $1$ of $\Gamma$ (resp. $\Gamma'$) is $l$ (resp. $l'$). Since $A$ and $A'$ are stably equivalent of Morita type, we have $r=r'$, $n=n'$ and $l=l'$.

\medskip
{\it Case 1: $\mathrm{char}(k)>0$.}

By Proposition \ref{prop:ZZ} we have an isomorphism of $k$-algebras $Z(A)/R(A)\cong Z(A')/R(A')$, where $R(A)$ (resp. $R(A')$) is the Reynolds ideal of $A$ (resp. $A'$). For each positive integer $i$, denote by $r_i$ (resp. $r'_i$) the number of vertices in $\Gamma$ (resp. $\Gamma'$) with multiplicity $i$. According to Corollary \ref{cor-center-quotient}, we have $r_2+l=r'_2+l'$ and $r_i=r'_i$ for each $i>2$. Since $l=l'$, we have $r_i=r'_i$ for each $i>1$. Since $r=r'$ and $r=\sum_{i\geq 1}r_i$ (resp. $r'=\sum_{i\geq 1}r'_i$), we also have $r_1=r'_1$. Therefore $\Gamma$ and $\Gamma'$ have the same multi-sets of multiplicities of vertices.

\medskip
{\it Case 2: $\mathrm{char}(k)=0$.}

Since $A$ and $A'$ are stably equivalent of Morita type, according to Proposition \ref{prop:P}, the $k$-algebras $Z^{\mathrm{st}}(A)$ and $Z^{\mathrm{st}}(A')$ are isomorphic. Then $Z^{\mathrm{st}}(A)/\mathrm{soc}(Z^{\mathrm{st}}(A))\cong Z^{\mathrm{st}}(A')/\mathrm{soc}(Z^{\mathrm{st}}(A'))$ as $k$-algebras. According to Proposition \ref{stable center of BGA}, the multi-sets of multiplicities of vertices which are larger or equal to $3$ of $\Gamma$ and $\Gamma'$ are the same. Since $\Gamma$ and $\Gamma'$ have the same number of vertices, to show that $\Gamma$ and $\Gamma'$ have the same multi-sets of multiplicities of vertices, it suffices to show that the sum of multiplicities of vertices of $\Gamma$ and $\Gamma'$ are equal.

According to \cite[Theorem 2]{AZ1} and \cite[Corollary 1.2]{LZZ}, the centers $Z(A)$ and $Z(A')$ have the same dimension. According to Proposition \ref{prop:center-of-stBA}, $$\mathrm{dim}_{k}(Z(A))=1+\sum_{i=1}^{r}m_i -r+l+n$$ and $$\mathrm{dim}_{k}(Z(A'))=1+\sum_{i=1}^{r'}m'_i -r'+l'+n'.$$ Since $r=r'$, $n=n'$ and $l=l'$, we have $\sum_{i=1}^{r}m_i=\sum_{i=1}^{r'}m'_i$.
\end{proof}

\begin{Thm}\label{cor-stm-der-BGA}
Two Brauer graph algebras are stably equivalent of Morita type if and only if they are derived equivalent.
\end{Thm}

\begin{proof}
According to \cite[Corollary 5.5]{Rickard1991}, two derived equivalent Brauer graph algebras are stably equivalent of Morita type.

Conversely, let $A$ and $A'$ be two Brauer graph algebras with associated Brauer graphs $\Gamma$ and $\Gamma'$ respectively, such that $A$ and $A'$ are stably equivalent of Morita type.

\medskip
{\it Case 1: $A$ is local.}

According to \cite[Theorem 2]{AZ1}, $A'$ is also local. Then both $\Gamma$ and $\Gamma'$ contain exactly one edge.

\medskip
{\it Case 1.1: $\Gamma$ is a loop with $1$ as the multiplicity of the unique vertex, or an edge with $2$ as the multiplicity of both vertices.}

It follows from \cite[Theorem 5.1]{S} that $A$ is domestic. According to \cite[Theorem 2.1]{ES}, its stable AR-quiver $\prescript{}{s}{\Gamma}_{A}$ contains a component of the form $\mathbb{Z}\widetilde{A}_{p,q}$. Since $A$ and $A'$ are stably equivalent, the stable AR-quiver $\prescript{}{s}{\Gamma}_{A'}$ of $A'$ also contains such a component. Therefore $A'$ is also domestic, which again follows from \cite[Theorem 2.1]{ES}. Since $\Gamma'$ contains exactly one edge, according to \cite[Theorem 5.1]{S} $\Gamma'$ is either a loop with $1$ as the multiplicity of the unique vertex or an edge with $2$ as the multiplicity of both vertices. If $\mathrm{ckar}(k)=2$, it follows from Lemma \ref{lem-not-st-eq} that $\Gamma$ and $\Gamma'$ are isomorphic, and therefore $A\cong A'$. If $\mathrm{ckar}(k)\neq 2$, by Example \ref{ex-isomorphic-BGAs} we also have $A\cong A'$. Therefore $A$ and $A'$ are derived equivalent.

\medskip
{\it Case 1.2: $\Gamma$ is neither a loop with $1$ as the multiplicity of the unique vertex nor an edge with $2$ as the multiplicity of both vertices.}

By the discussions in Case 1.1, we see that $\Gamma'$ is neither a loop with $1$ as the multiplicity of the unique vertex nor an edge with $2$ as the multiplicity of both vertices. By Theorem \ref{thm:ant's-stable-invariants} and Proposition \ref{thm-multi-sets-of-multiplicity}, it can be shown that $\Gamma$ and $\Gamma'$, and therefore $A$ and $A'$, are isomorphic. So $A$ and $A'$ are derived equivalent.

\medskip
{\it Case 2: $A$ is not local.}

By \cite[Theorem 2]{AZ1}, $A'$ is also not local. By Theorems \ref{thm:ant's-stable-invariants}, \ref{thm:derived-eqv-between-BGAs} and Proposition \ref{thm-multi-sets-of-multiplicity}, $A$ and $A'$ are derived equivalent.
\end{proof}

Combining Theorem \ref{cor-stm-der-BGA} and \cite[Theorem 4.4]{CLLX} (or Theorem \ref{thm:st.M-BGA}), we have

\begin{Cor}\label{cor:stm-implies-derived-equ}
Let $A$ and $B$ be two algebras without semisimple direct summands which are stably equivalent of Morita type. If $A$ is a Brauer graph algebra, then $A$ and $B$ are derived equivalent.
\end{Cor}

\section{Some further remarks}\label{sec:rmk}

In this section, following some idea from \cite{CKL} we discuss the liftability of simple-images of Morita type over a Brauer graph algebra $A$, and determine the orbits of simple-images of Morita type under the action of stable Picard group $\mathrm{StPic}(A)$.

Let $A$ be an algebra. The set $\mathrm{StPic}(A)$ of natural isomorphism classes $[\phi]$ of stable auto-equivalence of Morita type $\phi:A\text{-}\underline{\mathrm{mod}}\rightarrow A\text{-}\underline{\mathrm{mod}}$
forms a group under the composition of functors, which is called {\it the stable Picard group of $A$} (see \cite[Section 2]{CKL}).

By \cite{CKL}, a set of objects $\mathcal{S}$ in $A\text{-}\underline{\mathrm{mod}}$ is called a {\it simple-image of Morita type} over $A$ if there exists an algebra $B$ together with a stable equivalence of Morita type $F:B\text{-}\underline{\mathrm{mod}}\rightarrow A\text{-}\underline{\mathrm{mod}}$ which maps a complete set of representatives of non-isomorphic non-projective simple $B$-modules to $\mathcal{S}$ (we identify two simple-images of Morita type over $A$ if there exists a bijection between them such that the corresponding objects are isomorphic). If $A$ is a self-injective algebra without semisimple summands, according to \cite[Proposition 2.7]{CKL}, there exists a bijection between the Morita equivalence classes of algebras without semisimple summands which are stably equivalent of Morita type to $A$ and the orbits of simple-images of Morita type over $A$ under the action of $\mathrm{StPic}(A)$.

\begin{Def}\textnormal{(see \cite{CKL})}\label{def:liftable-sms}
Let $A$ be a self-injective algebra. A simple-image of Morita type $\mathcal{S}$ in $A\text{-}\underline{\mathrm{mod}}$ is said to be liftable if there exists an algebra $B$ together with a stable equivalence of Morita type $F:B\text{-}\underline{\mathrm{mod}}\rightarrow A\text{-}\underline{\mathrm{mod}}$ which can be lifted to a derived equivalence such that $F$ maps a complete set of representatives of non-isomorphic non-projective simple $B$-modules to $\mathcal{S}$.
\end{Def}

It is asked in \cite{CKL} that for which kind of algebra $A$, every simple-image of Morita type in $A\text{-}\underline{\mathrm{mod}}$ is liftable?  According to \cite[Theorem 4.1]{CKL}, this question has a positive answer for any indecomposable non simple representation-finite self-injective algebra. Since the positive answer puts a strong restriction on a given algebra, it seems useful to present a variant problem about the liftablilty of simple-images of Morita type as follows.

\begin{Prob}\label{prob}
Let $A$ be a self-injective algebra. When each $\mathrm{StPic}(A)$-orbit of simple-images of Morita type in $A\text{-}\underline{\mathrm{mod}}$ contains a liftable representative?
\end{Prob}

If $A$ is indecomposable and non simple, then the following lemma gives an equivalent characterization of above problem.

\begin{Lem}\label{lem:an-equivalent-characterization}
Let $A$ be an indecomposable non simple self-injective algebra.
\begin{itemize}
\item[$(1)$] Each simple-image of Morita type in $A\text{-}\underline{\mathrm{mod}}$ is liftable if and only if for each indecomposable non simple algebra $B$ and each stable equivalence of Morita type $F:B\text{-}\underline{\mathrm{mod}}\rightarrow A\text{-}\underline{\mathrm{mod}}$, $F$ can be lifted to a derived equivalence.
\item[$(2)$] Each $\mathrm{StPic}(A)$-orbit of simple-images of Morita type in $A\text{-}\underline{\mathrm{mod}}$ contains a liftable representative if and only if each indecomposable non simple algebra which is stably equivalent of Morita type to $A$ is also derived equivalent to $A$.
\end{itemize}
\end{Lem}

\begin{proof}
For an algebra $B$, we denote by $\mathcal{S}_{B}$ a complete set of representatives of non-isomorphic non-projective simple $B$-modules.

Suppose that each simple-image of Morita type in $A\text{-}\underline{\mathrm{mod}}$ is liftable. For an indecomposable non simple algebra $B$ together with a stable equivalence of Morita type $F:B\text{-}\underline{\mathrm{mod}}\rightarrow A\text{-}\underline{\mathrm{mod}}$, we denote by $\mathcal{S}=F(\mathcal{S}_{B})$. Since $\mathcal{S}$ is liftable, there exists an algebra $C$ together with a stable equivalence of Morita type $G:C\text{-}\underline{\mathrm{mod}}\rightarrow A\text{-}\underline{\mathrm{mod}}$ which can be lifted to a derived equivalence such that $G$ maps $\mathcal{S}_{C}$ to $\mathcal{S}$. We may assume that $C$ contains no semisimple direct summands. By \cite[Proposition 2.1 and Corollary 2.3]{L}, $C$ is indecomposable and self-injective. Note that $B$ is also self-injective, which again follows from \cite[Corollary 2.3]{L}. The composition $G^{-1}F:B\text{-}\underline{\mathrm{mod}}\rightarrow C\text{-}\underline{\mathrm{mod}}$ is a stable equivalence of Morita type which maps simples to simples. According to \cite[Theorem 2.1]{Linckelmann1996}, $G^{-1}F$ can be lifted to a Morita equivalence. Then $F=G(G^{-1}F)$ can be lifted to a derived equivalence.

Conversely, suppose that for each indecomposable non simple algebra $B$ and each stable equivalence of Morita type $F:B\text{-}\underline{\mathrm{mod}}\rightarrow A\text{-}\underline{\mathrm{mod}}$, $F$ can be lifted to a derived equivalence. Let $\mathcal{S}$ be a simple-image of Morita type in $A\text{-}\underline{\mathrm{mod}}$, which is the image of $\mathcal{S}_{C}$ under a stable equivalence of Morita type $G:C\text{-}\underline{\mathrm{mod}}\rightarrow A\text{-}\underline{\mathrm{mod}}$. We may assume that $C$ contains no semisimple direct summands. By \cite[Proposition 2.1 and Corollary 2.3]{L}, $C$ is indecomposable and self-injective. Then $G$ can be lifted to a derived equivalence and $\mathcal{S}$ is liftable. This shows $(1)$.

Suppose that each $\mathrm{StPic}(A)$-orbit of simple-images of Morita type in $A\text{-}\underline{\mathrm{mod}}$ contains a liftable representative. Let $B$ be an indecomposable non simple algebra which is stably equivalent of Morita type to $A$. Choose a stable equivalence of Morita type $F:B\text{-}\underline{\mathrm{mod}}\rightarrow A\text{-}\underline{\mathrm{mod}}$, and denote by $\mathcal{S}=F(\mathcal{S}_{B})$. There exists a stable auto-equivalence of Morita type $H$ of $A$ such that $H(\mathcal{S})$ is liftable. Then there exists an algebra $C$ together with a stable equivalence of Morita type $G:C\text{-}\underline{\mathrm{mod}}\rightarrow A\text{-}\underline{\mathrm{mod}}$ which can be lifted to a derived equivalence such that $G$ maps $\mathcal{S}_{C}$ to $H(\mathcal{S})$. We may assume that $C$ is indecomposable and self-injective. Since $F^{-1}H^{-1}G:C\text{-}\underline{\mathrm{mod}}\rightarrow B\text{-}\underline{\mathrm{mod}}$ is a stable equivalence of Morita type mapping simples to simples, it can be lifted to a Morita equivalence. Then $B$ is derived equivalent to $A$.

Conversely, suppose that each indecomposable non simple algebra which is stably equivalent of Morita type to $A$ is also derived equivalent to $A$. Let $\mathcal{S}$ be a simple-image of Morita type in $A\text{-}\underline{\mathrm{mod}}$. Choose a stable equivalence of Morita type $F:B\text{-}\underline{\mathrm{mod}}\rightarrow A\text{-}\underline{\mathrm{mod}}$ which maps $\mathcal{S}_{B}$ to $\mathcal{S}$. We may suppose that $B$ is indecomposable and self-injective. Therefore $B$ is derived equivalent to $A$, and we can choose a standard derived equivalence $\phi$ from $B$ to $A$, which induces a stable equivalence of Morita type $G:B\text{-}\underline{\mathrm{mod}}\rightarrow A\text{-}\underline{\mathrm{mod}}$. Denote by $H$ the stable auto-equivalence of Morita type $GF^{-1}$ of $A$. Then $H(\mathcal{S})=G(\mathcal{S}_{B})$ is a liftable simple-image of Morita type. This shows $(2)$.
\end{proof}

Combining Corollary \ref{cor:stm-implies-derived-equ} and Lemma \ref{lem:an-equivalent-characterization}, we have

\begin{Prop}\label{prop:orbit-contains-liftable}
If $A$ is a Brauer graph algebra, then each $\mathrm{StPic}(A)$-orbit of simple-images of Morita type in $A\text{-}\underline{\mathrm{mod}}$ contains a liftable representative.
\end{Prop}

\begin{Rem1}
\begin{itemize}
\item[$(1)$] In \cite[Example 6.4]{LL}, it shows that there exists a stable auto-equivalence of Morita type over some Brauer graph algebra that is not lifted to a derived equivalence. Then according to Lemma \ref{lem:an-equivalent-characterization}, there exists a simple-image of Morita type over a Brauer graph algebra that is not liftable.
\item[$(2)$] According to \cite[Remark 2.10(3)]{CKL}, there exists some self-injective $A$ such that Problem \ref{prob} has a negative answer.
\end{itemize}
\end{Rem1}

When $A=A_{\Gamma}$ is a Brauer graph algebra, it follows from \cite[Theorem 4.4]{CLLX} (or Theorem \ref{thm:st.M-BGA}) that each algebra $B$ without semisimple summands which is stably equivalent of Morita type to $A$ is Morita equivalent to a Brauer graph algebra $A_{\Gamma'}$. By Theorem \ref{cor-stm-der-BGA}, $A_{\Gamma'}$ is derived equivalent to $A$. According to Proposition \ref{prop-uniqueness-of-BG}, two Brauer graph algebras $A_{\Gamma_1}$, $A_{\Gamma_2}$ are isomorphic if and only if the associated Brauer graphs $\Gamma_1$, $\Gamma_2$ are isomorphic, except when $\Gamma_1$ is a loop with $1$ as the multiplicity of the unique vertex or an edge with $2$ as the multiplicity of both vertices. Therefore we have

\begin{Prop}\label{prop:classification-of-orbits}
Let $A=A_{\Gamma}$ be a Brauer graph algebra. If $\Gamma$ is neither a loop with $1$ as the multiplicity of the unique vertex nor an edge with $2$ as the multiplicity of both vertices, then the orbits of simple-images of Morita type over $A$ under $\mathrm{StPic}(A)$ are one-to-one correspond with the isomorphism classes of Brauer graphs $\Gamma'$ such that $A_{\Gamma'}$ and $A_{\Gamma}$ are derived equivalent.
\end{Prop}

Let $A_{\Gamma_1},A_{\Gamma_2},\cdots,A_{\Gamma_r}$ be all the Brauer graph algebras which are derived equivalent to $A$, and let $F_i:A_{\Gamma_i}\text{-}\underline{\mathrm{mod}}\rightarrow A\text{-}\underline{\mathrm{mod}}$ be a stable equivalence which lifts to a derived equivalence for each $i=1,2,\cdots,r$. Then $\{F_1(\mathcal{S}_1),\cdots,F_r(\mathcal{S}_r)\}$ forms a complete set of representatives of simple-images of Morita type over $A$ which belongs to different $\mathrm{StPic}(A)$-orbits, where $\mathcal{S}_i$ is a complete set of representatives of non-isomorphic simple $A_{\Gamma_i}$-module. We will illustrate this observation by the following example.

\begin{Ex1}
Let $A=A_{\Gamma}$ be a Brauer graph algebra given by the following Brauer graph
$$
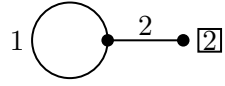
\begin{tikzpicture}
\draw (0,0) circle (0.5);
\fill (0.5,0) circle (0.5ex);
\draw (0.5,0) -- (1.5,0);
\fill (1.5,0) circle (0.5ex);
\node at(1.85,0) {$2$};
\draw (1.7,-0.15) rectangle (2,0.15);
\node at(-0.7,0) {$1$};
\node at(1,0.2) {$2$};
\end{tikzpicture}.$$
Then $A$ is given by the quiver
$$
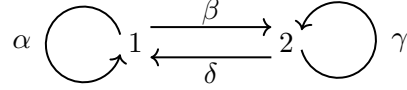
\begin{tikzpicture}
\node at(0,0) {$1$};
\node at(2,0) {$2$};
\draw[->] (0.2,0.2) -- (1.8,0.2);
\draw[->] (1.8,-0.2) -- (0.2,-0.2);
\draw[->] (-0.2,0.1) arc (15:345:0.5);
\draw[->] (2.2,-0.1) arc (195:525:0.5);
\node at(-1.5,0) {$\alpha$};
\node at(3.5,0) {$\gamma$};
\node at(1,0.4) {$\beta$};
\node at(1,-0.4) {$\delta$};
\end{tikzpicture}$$
with relations $\alpha^2=\gamma\beta=\delta\gamma=\beta\delta=0$, $\delta\beta\alpha=\alpha\delta\beta$, $\gamma^2=\beta\alpha\delta$.  The indecomposable projective (left) $A$-modules have the following structures
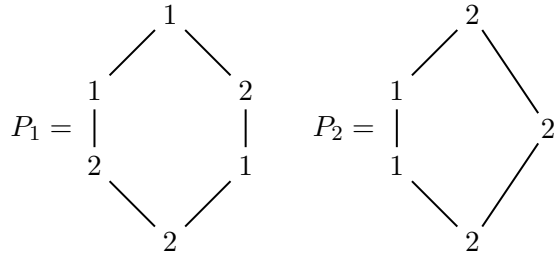
\begin{center}
\tikzset{every picture/.style={line width=0.75pt}}
\begin{tikzpicture}
\node at(-1.7,1.5) {$P_1=$};
\node at(0,3) {$1$};
\node at(-1,2) {$1$};
\node at(1,2) {$2$};
\node at(-1,1) {$2$};
\node at(1,1) {$1$};
\node at(0,0) {$2$};
\draw    (-0.2,0.2) -- (-0.8,0.8) ;
\draw    (0.2,0.2) -- (0.8,0.8) ;
\draw    (-1,1.25) -- (-1,1.75) ;
\draw    (1,1.25) -- (1,1.75) ;
\draw    (-0.8,2.2) -- (-0.2,2.8) ;
\draw    (0.8,2.2) -- (0.2,2.8) ;
\node at(2.3,1.5) {$P_2=$};
\node at(4,0) {$2$};
\node at(3,1) {$1$};
\node at(3,2) {$1$};
\node at(4,3) {$2$};
\node at(5,1.5) {$2$};
\draw    (3.2,2.2) -- (3.8,2.8) ;
\draw    (3,1.25) -- (3,1.75) ;
\draw    (3.8,0.2) -- (3.2,0.8) ;
\draw    (4.13,2.8) -- (4.87,1.7) ;
\draw    (4.13,0.2) -- (4.87,1.3) ;
\end{tikzpicture}.
\end{center}

Denote by $T_1$ the complex
$$0\rightarrow P_2\oplus P_2\xrightarrow{(\beta,\beta\alpha)} P_1\rightarrow 0,$$
where $P_1$ is the degree $0$ term of this complex, and denote by $T_2=P_2[1]$. Then $T=T_1\oplus T_2$ is an Okuyama tilting complex and $B=\mathrm{End}_{K^b(A)}(T)^{op}$ is a Brauer graph algebra given by the following Brauer graph
$$
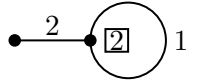
\begin{tikzpicture}
\draw (2,0) circle (0.5);
\fill (0.5,0) circle (0.5ex);
\draw (0.5,0) -- (1.5,0);
\fill (1.5,0) circle (0.5ex);
\node at(1.85,0) {$2$};
\draw (1.7,-0.15) rectangle (2,0.15);
\node at(2.7,0) {$1$};
\node at(1,0.2) {$2$};
\end{tikzpicture}.$$
According to Theorem \ref{thm:derived-eqv-between-BGAs}, $A$ and $B$ are the only Brauer graph algebras which are derived equivalent to $A$. This also follows from the fact that $A$ is tilting-discrete; see~\cite{AAC}.

Denote by $F:B\text{-}\underline{\mathrm{mod}}\rightarrow A\text{-}\underline{\mathrm{mod}}$ the induced stable equivalence of the derived equivalence $D^b(B\text{-}\mathrm{mod})\rightarrow D^b(A\text{-}\mathrm{mod})$ which is induced by the tilting complex $T$. According to Okuyama's Lemma (see \cite[Section 5]{Dugas2015}), $F$ takes the set of simple $B$-modules to the set
$$\begin{tikzpicture}
\node at(-1,0) {$\mathcal{S}'=$};
\node[rotate = 270] at (-0.2,0) {$\underbrace{\hspace{1cm}}$};
\node at(0,0) {$1$};
\node at(0.2,-0.2) {,};
\node at(0.5,0.2) {$2$};
\node at(0.5,-0.2) {$2$};
\node[rotate = 90] at (0.8,0) {$\underbrace{\hspace{1cm}}$};
\node at(1.2,-0.4) {.};
\end{tikzpicture}$$
Denote by $\mathcal{S}$ the set of simple $A$-modules. Then $\{\mathcal{S},\mathcal{S}'\}$ forms a complete set of representatives for the orbits of the action of $\mathrm{StPic}(A)$ on the set of simple-images of Morita type over $A$. 
\end{Ex1}

\newpage
\appendix
\section{A new proof of the fact that Brauer graph algebras are closed under stable equivalence of Morita type}\label{sec-appendix}

In this appendix we will give another proof of the proposition that Brauer graph algebras are closed under stable equivalence of Morita type by using invariants. Note that there has been a proof in \cite{CLLX} but here we use a different method and the proof is more straightforward. Throughout this appendix we fix $k$ to be an algebraically closed field.

\begin{definition}\textnormal{(\cite[Definition 1]{AZ1})}
		Let $Q$ be a quiver and $I$ an admissible ideal of $kQ$. A self-injective algebra $A$ is said to be stably biserial if $A$ is isomorphic to $kQ/I$, where $Q$ and $I$ satisfy the following conditions:
		\begin{enumerate}
			\item for each vertex $i\in Q$, the number of outgoing and incoming arrows are less than or equal to $2$;
			
			\item for each arrow $\alpha\in Q$, there is at most one arrow $\beta\in Q$ such that $\alpha\beta\notin \alpha\rad(A)\beta+\soc(A)$;
			
			\item for each arrow $\alpha\in Q$, there is at most one arrow $\gamma\in Q$ such that $\gamma\alpha\notin \gamma\rad(A)\alpha+\soc(A)$.
		\end{enumerate}
	\end{definition}

An algebra $A$ is called {\it special biserial} if it is isomorphic to an algebra of the form $\qa$ where $kQ$ is a path algebra and $I$ is an admissible ideal such that the following properties hold:
	
	\begin{enumerate}
		\item At every vertex $i$ in $Q$, there are at most two arrows starting at $i$ and there are at most two arrows ending at $i$.
		
		\item For every arrow $\alpha$ in $Q$, there exists at most one arrow $\beta$ such that $\beta\alpha\notin I$ and there exists at most one arrow $\gamma$ such that $\alpha\gamma\notin I$.
	\end{enumerate}

Thus, according to the above definitions, self-injective special biserial algebras form a special subclass of stably biserial algebras. In fact, algebras that are stably equivalent to self-injective special biserial algebras have the following characterization.

\begin{theorem}\textnormal{(\cite[Theorem 1]{AZ1})}\label{thm:sta-to-BGA=StB}
		Let $A$ be an indecomposable self-injective special biserial $k$-algebra which is not isomorphic to the Nakayama algebra with $\rad^2(A)=0$. If $B$ is a basic algebra stably equivalent to $A$, then $B$ is stably biserial.
	\end{theorem}

Note that every symmetric stably biserial algebra can be represented by a Brauer graph $(\Gamma, m)$ together with a special set $\mathcal{L}$ of loops in $Q_\Gamma$. This result was proved in \cite[Section 5]{AZ1}, and a clearer exposition can be found in \cite[Theorem 2.4]{AZ2}. For various notations about Brauer graph algebra used below, we refer readers to Subsection \ref{subsec:def-BGA}.

\begin{Thm}\label{thm:sym-StBA}
	Any indecomposable symmetric stably biserial algebra $A$ has a presentation $kQ/I$, where $Q=Q_\Gamma$ is a connected quiver associated with a Brauer graph $(\Gamma,m)$, and the ideal of relations $I$ is generated by
	\begin{enumerate}
		\item $\alpha\beta=0$ for all $\alpha,\beta\in Q_1$, $\beta\neq \sigma(\alpha)$ and $\alpha\notin\mathcal{L}$;
		\item $C_\alpha^{m(C_\alpha)}=C_\beta^{m(C_\beta)}$ for all $\alpha,\beta\in Q_1$ with $t(\alpha)=t(\beta)$;
		\item $\alpha_i^2=t_{\alpha_i} C_{\alpha_i}^{m(C_{\alpha_i})}$ for each $\alpha_i\in\mathcal{L}$;
		\item $C_\alpha^{m(C_\alpha)}\beta=0$ for all $\alpha,\beta\in Q_1$,
	\end{enumerate}
	where $\mathcal{L}=\{\alpha_1,\cdots,\alpha_n\}\subseteq Q_1$ such that each $\alpha_i$ is a loop with $\sigma(\alpha_i)\neq \alpha_i$, and $t_{\alpha_i}\in k^*$.

	Moreover, when $\mathrm{char}(k)\neq 2$, any symmetric stably biserial algebra is isomorphic to an algebra $\qa$ as above with $\mathcal{L}=\varnothing$, in which case it is a Brauer graph algebra.
\end{Thm}

Note that the fact that $Q_\Gamma$ contains a deformed loop implies that $\Gamma$ contains a loop corresponding to a face of perimeter $1$.

\begin{Rem1}\label{remark-center-of-SSBA}
If $B$ is a symmetric stably biserial algebra defined by the Brauer graph $(\Gamma, m)$ and a set $\mathcal{L}$ of deformed loops, then according to \cite[Proposition 4.1]{AZ2}, $Z(B)$ is also generated by $1$ together with the elements of type $(a)$, $(b)$, $(c)$ in Proposition \ref{prop:center-of-stBA} as a $k$-vector space. Moreover, if $A$ is the Brauer graph algebra defined by the Brauer graph $(\Gamma, m)$, then $Z(A)/\mathrm{soc}(Z(A))$ and $Z(B)/\mathrm{soc}(Z(B))$ (resp. $Z(A)/R(A)$ and $Z(B)/R(B)$) are isomorphic as $k$-algebras, where $R(A)$ and $R(B)$ denote the Reynolds ideals of $A$ and $B$ respectively (for the definition of Reynolds ideal, see Subsection \ref{subsec:stb-and-center}).
\end{Rem1}

\begin{Rem1}\label{remark-center-quotient}
Let $A$ (resp. $A'$) be a symmetric stably biserial algebra defined by the Brauer graph $(\Gamma, m)$ and a set $\mathcal{L}$ of deformed loops (resp. by the Brauer graph $(\Gamma',m')$ and a set $\mathcal{L}'$ of deformed loops). Denote by $r_i$ (resp. $r'_i$) the number of vertices in $\Gamma$ (resp. $\Gamma'$) with multiplicity $i$, and denote by $l$ (resp. $l'$) the number of faces of $\Gamma$ (resp. $\Gamma'$) of perimeter $1$. If $A$ and $A'$ are stably equivalent of Morita type and $k$ is of positive characteristic, then according to Proposition \ref{prop:ZZ}, Corollary \ref{cor-center-quotient} and Remark \ref{remark-center-of-SSBA}, we have $r_i=r'_i$ for $i>2$ and $r_2+l=r'_2+l'$.
\end{Rem1}

By comparing the tubes in the stable Auslander-Reiten quiver of a symmetric stably biserial algebra with the faces of its corresponding Brauer graph, Antipov and Zvonareva established in \cite[Section 4.2]{AZ2} the following result.

\begin{Prop}\label{prop:stbA-tube-and-face}
	Let $A$ be a representation-infinite symmetric stably biserial algebra with Brauer graph $\Gamma$. Then, for any face of $\Gamma$ of perimeter $p > 2$, the structure of the stable Auslander-Reiten quiver $\prescript{}{s}{\Gamma}_{A}$ of $A$ reflects this combinatorial data. If $p$ is odd, then $\prescript{}{s}{\Gamma}_{A}$ contains a tube of rank $p$ that is invariant under the syzygy functor $\Omega_A$. If $p$ is even, then $\prescript{}{s}{\Gamma}_{A}$ contains a pair of tubes, each of rank $p/2$, which are exchanged by $\Omega_A$. Moreover, each tube of rank $p\geq 2$ of $\prescript{}{s}{\Gamma}_{A}$ is corresponded to some face of $\Gamma$ under this correspondence.
\end{Prop}

\begin{Cor}\label{cor-stbA-tube-and-face}
Let $A$ and $A'$ be representation-infinite symmetric stably biserial algebras with Brauer graph $\Gamma$, $\Gamma'$ respectively. If $A$ and $A'$ are stably equivalent, then for each positive integer $p>2$, $\Gamma$ and $\Gamma'$ have the same number of faces of perimeter $p$.
\end{Cor}

\begin{proof}
Since both $A$ and $A'$ are representation-infinite, $\rad^2(A)\neq 0$ and $\rad^2(A')\neq 0$. Then a stable equivalent from $A$ to $A'$ induces an isomorphism between the stable Auslander-Reiten quiver $\prescript{}{s}{\Gamma}_{A}$ of $A$ and the stable Auslander-Reiten quiver $\prescript{}{s}{\Gamma}_{A'}$ of $A'$. Moreover, a tube of $\prescript{}{s}{\Gamma}_{A}$ which is invariant under the syzygy functor $\Omega_A$ (resp. a pair of tubes of $\prescript{}{s}{\Gamma}_{A}$ which are exchanged by $\Omega_A$) corresponds to a tube of $\prescript{}{s}{\Gamma}_{A'}$ which is invariant under the syzygy functor $\Omega_{A'}$ (resp. a pair of tubes of $\prescript{}{s}{\Gamma}_{A'}$ which are exchanged by $\Omega_{A'}$), and the rank of the corresponding tubes are equal. By the correspondence between the faces of Brauer graph of perimeter $>2$ and the tubes in stable Auslander-Reiten quiver with rank $\geq 2$ in Proposition \ref{prop:stbA-tube-and-face}, $\Gamma$ and $\Gamma'$ have the same number of faces of perimeter $p$ for each positive integer $p>2$.
\end{proof}

The following lemma is standard; see, for example, \cite{TW, LZZ}. For the convenience of the reader, we include a proof here.

\begin{Lem}\label{lem-invariant-factors}
Suppose that $A$ and $B$ are two self-injective algebras whose stable categories are triangulated equivalent, and denote by $C(A)$ and $C(B)$ the Cartan matrices of $A$ and $B$ respectively. If $A$ and $B$ have the same number of isomorphism classes of simple modules, then as matrices over $\mathbb{Z}$, $C(A)$ and $C(B)$ have the same invariant factors. In particular, $C(A)$ and $C(B)$ have the same rank.
\end{Lem}

\begin{proof}
Denote by $\text{K}_{0}(A\text{-}\underline{\mathrm{mod}})$ and $\text{K}_{0}(B\text{-}\underline{\mathrm{mod}})$ the stable Grothendieck groups of $A$ and $B$ respectively. It is well-known that there exist exact sequences
$$\mathbb{Z}^{n}\xrightarrow{C(A)}\mathbb{Z}^{n}\rightarrow
\text{K}_{0}(A\text{-}\underline{\mathrm{mod}})\rightarrow 0$$
and
$$\mathbb{Z}^{n}\xrightarrow{C(B)}\mathbb{Z}^{n}\rightarrow
\text{K}_{0}(B\text{-}\underline{\mathrm{mod}})\rightarrow 0,$$
where $n$ is the number of isomorphism classes of simple $A$-modules. Denote by $d_1\mid d_2\mid\cdots\mid d_r$ (resp. $e_1\mid e_2\mid\cdots\mid e_l$) the invariant factors of $C(A)$ (resp. $C(B)$) as a matrix over $\mathbb{Z}$, and suppose that $d_1,d_2,\cdots,d_{r'}=1$ and $d_{r'+1}>1$ (resp. $e_1,e_2,\cdots,e_{l'}=1$ and $e_{l'+1}>1$). Then $$\text{K}_{0}(A\text{-}\underline{\mathrm{mod}})\cong\mathbb{Z}^{n-r}\oplus\mathbb{Z}/d_1\mathbb{Z}
\oplus\mathbb{Z}/d_2\mathbb{Z}\oplus\cdots\oplus\mathbb{Z}/d_r\mathbb{Z}$$
and
$$\text{K}_{0}(B\text{-}\underline{\mathrm{mod}})\cong\mathbb{Z}^{n-l}\oplus\mathbb{Z}/e_1\mathbb{Z}
\oplus\mathbb{Z}/e_2\mathbb{Z}\oplus\cdots\oplus\mathbb{Z}/e_l\mathbb{Z}.$$
The free part of $\text{K}_{0}(A\text{-}\underline{\mathrm{mod}})$ (resp. $\text{K}_{0}(B\text{-}\underline{\mathrm{mod}}))$ is $\mathbb{Z}^{n-r}$ (resp. $\mathbb{Z}^{n-l}$), and the torsion part of $\text{K}_{0}(A\text{-}\underline{\mathrm{mod}})$ (resp. $\text{K}_{0}(B\text{-}\underline{\mathrm{mod}}))$ is
$\mathbb{Z}/d_{r'+1}\mathbb{Z}
\oplus\mathbb{Z}/d_{r'+2}\mathbb{Z}\oplus\cdots\oplus\mathbb{Z}/d_r\mathbb{Z}$
(resp. $\mathbb{Z}/e_{l'+1}\mathbb{Z}
\oplus\mathbb{Z}/e_{l'+2}\mathbb{Z}\oplus\cdots\oplus\mathbb{Z}/e_l\mathbb{Z}$). Since the stable categories of $A$ and $B$ are triangulated equivalent, $\text{K}_{0}(A\text{-}\underline{\mathrm{mod}})$ and $\text{K}_{0}(B\text{-}\underline{\mathrm{mod}})$ are isomorphic. Then we have $n-r=n-l$, $r-r'=l-l'$, and $d_{r'+i}=e_{l'+i}$ for $i=1,2,\cdots,r-r'$. Therefore $C(A)$ and $C(B)$ have the same invariant factors, and in particular, the rank of $C(A)$ and $C(B)$ are equal.
\end{proof}

\begin{Def}\textnormal{(see \cite[Section 3]{AZ2})}\label{def-caterpillar}
A Brauer graph algebra $A=A_{\Gamma}$ is said to be a caterpillar if its defining Brauer graph $\Gamma$ (as a ribbon graph) is of the form
$$
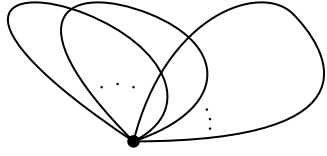
\begin{tikzpicture}[x=12pt,y=12pt,yscale=1,xscale=1]
\fill (0,0) circle (0.5ex);
\draw    (0,0) .. controls (-6,4) and (-4,5) .. (-1.5,4)
.. controls (1,3) and (2,1) .. (0,0);
\draw    (0,0) .. controls (-4,4) and (-2,5) .. (0.5,4)
.. controls (3,3) and (3,1) .. (0,0);
\draw    (0,0) .. controls (1,4) and (4,5) .. (5,4)
.. controls (6,3) and (8,0) .. (0,0);
\fill (-1,1.7) circle (0.1ex);
\fill (-0.5,1.8) circle (0.1ex);
\fill (0,1.7) circle (0.1ex);
\fill (2.3,1) circle (0.1ex);
\fill (2.4,0.7) circle (0.1ex);
\fill (2.4,0.4) circle (0.1ex);
\end{tikzpicture}$$
or of the form
$$
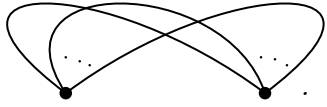
\begin{tikzpicture}[x=15pt,y=15pt,yscale=1,xscale=1]
\fill (0,0) circle (0.5ex);
\fill (5,0) circle (0.5ex);
\node at(6,0) {.};
\fill (0,0.9) circle (0.1ex);
\fill (0.35,0.8) circle (0.1ex);
\fill (0.6,0.7) circle (0.1ex);
\fill (4.9,0.9) circle (0.1ex);
\fill (5.25,0.85) circle (0.1ex);
\fill (5.55,0.7) circle (0.1ex);
\draw    (0,0) .. controls (-4,3) and (1,3) .. (5,0) ;
\draw    (0,0) .. controls (-2,3) and (4,3) .. (5,0) ;
\draw    (0,0) .. controls (4,3) and (9,3) .. (5,0) ;
\end{tikzpicture}$$
\end{Def}

For an algebra $\Lambda$, denote by $\mathrm{Out}^{0}(\Lambda)$ the identity component of the group of outer automorphisms $\mathrm{Out}(\Lambda)=\mathrm{Aut}(\Lambda)/\mathrm{Inn}(\Lambda)$ of $\Lambda$. Here $\mathrm{Aut}(\Lambda)$ is the group of automorphisms of $\Lambda$ and $\mathrm{Inn}(\Lambda)=\{f_a\mid a$ is a unit in $\Lambda\}$ is the group of inner automorphisms of $\Lambda$, where $f_a(x)=axa^{-1}$. If $\Lambda$ and $\Lambda'$ are self-injective and are stably equivalent of Morita type, then it follows from \cite[Th\'eor\`eme 4.3]{Rou} that $\mathrm{Out}^{0}(\Lambda)$ and $\mathrm{Out}^{0}(\Lambda')$ are isomorphic as algebraic groups.

\begin{Thm}\textnormal{(\cite[Theorem 1.1]{AZ2})}\label{thm-maximal-tori}
Let $A$ be a non-local symmetric stably biserial algebra over $k$ ($\mathrm{char}(k)=2$) or a non-local Brauer graph algebra over $k$ ($\mathrm{char}(k)\neq 2$), which is not a caterpillar. Let $\Gamma$
be the Brauer graph of $A$, $V(\Gamma)$ and $E(\Gamma)$ be the number of vertices and edges of $\Gamma$ respectively, and $d$ be the number of deformed loops in $A$ ($d=0$ if $A$ is a Brauer graph algebra). Then the rank of the maximal torus in $\mathrm{Out}^{0}(A)$ is $E(\Gamma)-V(\Gamma)-d+2$.
\end{Thm}

\begin{Thm}\textnormal{(\cite[Theorem 4.4]{CLLX})}\label{thm:st.M-BGA}
 Let $A$ be a Brauer graph algebra. Then for any basic algebra $B$ without semisimple summands, if $B$ and $A$ are stably equivalent of Morita type, then $B$ is also a Brauer graph algebra.
\end{Thm}

\begin{proof} 
By \cite[Corollary 2.4 and Proposition 2.1]{L} and Theorem \ref{thm:sta-to-BGA=StB}, $B$ is indecomposable and symmetric stably biserial. If $\mathrm{char}(k)\neq 2$, then according to Theorem \ref{thm:sym-StBA}, $B$ is a Brauer graph algebra. Therefore without loss of generality, we may assume that the field characteristic is $2$. Assume that the Brauer graph associated with $A$ is $(\Gamma, m)$, and let $B$ be a symmetric stably biserial algebra defined by the Brauer graph $(\Gamma', m')$ and a set $\mathcal{L}$ of deformed loops. According to \cite[Theorem 2]{AZ1}, $A$ and $B$ have the same number of isomorphism classes of simple modules. It follows that the number of edges of $\Gamma$ and $\Gamma'$ are equal.

\medskip
{\it Step 1: Suppose that $A$ is local.}

We assume that $B$ is not a Brauer graph algebra. Then $\Gamma'$ is a loop. Suppose that the multiplicity of the unique vertex in $\Gamma'$ is $m'$. Then the Cartan matrix $C(B)$ of $B$ is
\[\begin{pmatrix}
4m'
\end{pmatrix}, \]
which is a zero matrix over $k$. By Lemma \ref{lem-invariant-factors}, the Cartan matrix $C(A)$ of $A$ is also
\[\begin{pmatrix}
4m'
\end{pmatrix}, \]
which is a zero matrix over $k$. Since $A$ and $B$ are symmetric, the dimension of the projective center $Z^{\mathrm{pr}}(A)$ (resp. $Z^{\mathrm{pr}}(B)$) of $A$ (resp. $B$) is equal to the rank of $C(A)$ (resp. $C(B)$) as a matrix over $k$. Then the projective centers $Z^{\mathrm{pr}}(A)$, $Z^{\mathrm{pr}}(B)$ are equal to zero, and by Proposition \ref{prop:P} the centers of $A$ and $B$ are isomorphic as $k$-algebras.

\medskip
{\it Step 1.1: Suppose that $m'=1$.}

Since $m'=1$, $B$ is commutative and $Z(B)=B$. If $\Gamma$ is a loop and the multiplicity of the unique vertex in $\Gamma$ is $m\geq 2$, then a calculation shows that
$$Z(A)\cong k[X_1,X_2,X_3,X_4]/(X_{1}^{m},X_{2}^{2},X_{3}^{2},X_{4}^{2},X_i X_j (i\neq j)).$$
Since $Z(A)$ is not symmetric and $Z(B)$ is symmetric, they are not isomorphic, a contradiction. Therefore either $\Gamma$ is a loop such that the multiplicity of the unique vertex in $\Gamma$ is $1$, or $\Gamma$ is an ordinary edge. In both cases $A$ is commutative, and therefore $A\cong Z(A)\cong Z(B)\cong B$. Then $B$ is a Brauer graph algebra, which contradicts our assumption.

\medskip
{\it Step 1.2: Suppose that $m'>1$.}

A calculation shows that
$$Z(B)\cong k[X_1,X_2,X_3,X_4]/(X_{1}^{m'},X_{2}^{2},X_{3}^{2},X_{4}^{2},X_i X_j (i\neq j)),$$
which is not symmetric. If $\Gamma$ is a loop such that the multiplicity of the unique vertex in $\Gamma$ is $1$, or if $\Gamma$ is an ordinary edge, then $A$ is commutative and $Z(A)$ is symmetric, which is not isomorphic to $Z(B)$. Then $\Gamma$ is a loop and the multiplicity of the unique vertex in $\Gamma$ is $m\geq 2$. Since
$$Z(A)\cong k[X_1,X_2,X_3,X_4]/(X_{1}^{m},X_{2}^{2},X_{3}^{2},X_{4}^{2},X_i X_j (i\neq j))$$
is isomorphic to $Z(B)$, $m=m'$. We have that
$$A\cong k\langle\alpha,\beta\rangle/\langle(\beta\alpha)^m-(\alpha\beta)^m,\alpha^2,\beta^2\rangle$$
and
$$B\cong k\langle\alpha,\beta\rangle/\langle(\beta\alpha)^m-(\alpha\beta)^m,\alpha^2-t(\beta\alpha)^m,\beta^2-t'(\alpha\beta)^m,
\alpha(\beta\alpha)^m,\beta(\alpha\beta)^m\rangle$$
with $t,t'\in k$. Since $B$ is not a Brauer graph algebra, we may assume that $t\neq 0$. A calculation shows that $\mathrm{dim}_{k}\mathrm{HH}^{1}(A)=m+7$ and
\begin{equation*}
\mathrm{dim}_{k}\mathrm{HH}^{1}(B)=\begin{cases}
m+4, \text{ if } t'\neq 0 \text{ and } m \text{ is odd}; \\
m+5, \text{ if } t'\neq 0 \text{ and } m \text{ is even, or if } t'=0 \text{ and } m \text{ is odd}; \\
m+6, \text{ if } t'=0 \text{ and } m \text{ is even}.
\end{cases}
\end{equation*}
For the method of computation, we refer the reader to~\cite{RSS,LX}. Thus $\mathrm{HH}^1(A)$ and $\mathrm{HH}^1(B)$ have different dimensions. Since positive-degree Hochschild cohomology is invariant under stable equivalences of Morita type between finite-dimensional self-injective algebras~\cite{P2001,LXi05}, the algebras $A$ and $B$ cannot be stably equivalent of Morita type, a contradiction.

\medskip
{\it Step 2: Suppose that $A$ is not local.}

\medskip
{\it Step 2.1: Suppose that $A$ is a caterpillar.}

Suppose that $\Gamma$ has $n(\geq 2)$ edges. If $\Gamma$ has exactly one vertex, then $\Gamma$ contains two faces, each has perimeter $n$ if $n$ is odd, and contains one face of perimeter $2n$ if $n$ is even. Then each face of $\Gamma$ has perimeter $>2$. By Corollary \ref{cor-stbA-tube-and-face}, $\Gamma$ and $\Gamma'$ have the same multi-set of perimeters of faces of perimeter larger than $2$. Moreover, since $\Gamma$ and $\Gamma'$ have the same number of edges, the sum of perimeters of all faces of $\Gamma$ and $\Gamma'$ are equal. Therefore $\Gamma'$ has no faces of perimeter $\leq 2$. Then $\mathcal{L}=\varnothing$ and $B$ is a Brauer graph algebra.

If $\Gamma$ has exactly two vertex, then $\Gamma$ contains two faces, each has perimeter $n$ if $n$ is even, and contains one face of perimeter $2n$ if $n$ is odd. When $n\neq 2$, each face of $\Gamma$ has perimeter $>2$, and similarly we have that $B$ is a Brauer graph algebra. When $n=2$, $\Gamma$ has two faces, each has perimeter $2$. If $B$ is not a Brauer graph algebra, then $\Gamma'$ is of the form
$$
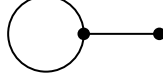
\begin{tikzpicture}
\draw (0,0) circle (0.5);
\fill (0.5,0) circle (0.5ex);
\draw (0.5,0) -- (1.5,0);
\fill (1.5,0) circle (0.5ex);
\end{tikzpicture}$$
or of the form
$$
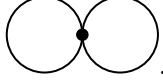
\begin{tikzpicture}
\draw (0,0) circle (0.5);
\draw (1,0) circle (0.5);
\fill (0.5,0) circle (0.5ex);
\end{tikzpicture}.$$
If $\Gamma'$ belongs to the first case, then it has a face of perimeter $3$, which contradicts the fact that $\Gamma$ and $\Gamma'$ have the same multi-set of perimeters of faces of perimeter larger than $2$. If $\Gamma'$ belongs to the second case, denote by $q_1$, $q_2$ the multiplicity of the two vertices of $\Gamma$ respectively and denote by $p$ the multiplicity of the unique vertex of $\Gamma'$. Then the Cartan matrix $C(A)$ of $A$ is \[ C(A)= \begin{pmatrix}
q_1+q_2 & q_1+q_2  \\
q_1+q_2 & q_1+q_2
\end{pmatrix} \]
and the Cartan matrix $C(B)$ of $B$ is \[ C(B)= \begin{pmatrix}
4p & 4p \\
4p & 4p
\end{pmatrix}. \]
By Lemma \ref{lem-invariant-factors}, we have $4p=q_1+q_2$. Since char$(k)=2$, both $C(A)$ and $C(B)$ are zero matrices over $k$, and the projective centers $Z^{\mathrm{pr}}(A)$, $Z^{\mathrm{pr}}(B)$ are equal to zero. By Proposition \ref{prop:P}, the centers of $A$ and $B$ are isomorphic as $k$-algebras. By Proposition \ref{remark-center-of-stBA} and Remark \ref{remark-center-of-SSBA},
$$Z(A)/(\mathrm{soc}(Z(A)))\cong k[X,Y]/(X^{q_1},Y^{q_2},XY)$$
and
\begin{equation*}
Z(B)/(\mathrm{soc}(Z(B)))\cong
\begin{cases}
k[X]/(X^{p-1}), \text{ if } p\geq 2; \\
k, \text{ if } p=1.
\end{cases}
\end{equation*}
Since $\mathrm{dim}_{k}(Z(A)/(\mathrm{soc}(Z(A))))=q_1+q_2-1=4p-1$ and
\begin{equation*}
\mathrm{dim}_{k}(Z(B)/(\mathrm{soc}(Z(B))))=
\begin{cases}
p-1, \text{ if } p\geq 2; \\
1, \text{ if } p=1,
\end{cases}
\end{equation*}
$Z(A)/(\mathrm{soc}(Z(A)))$ and $Z(B)/(\mathrm{soc}(Z(B)))$ have different dimensions, which are not isomorphic, a contradiction.

\medskip
{\it Step 2.2: Suppose that $A$ is not a caterpillar.}

Denote by $E(\Gamma)$ (resp. $E(\Gamma')$) the number of edges in $\Gamma$ (resp. $\Gamma'$), $V(\Gamma)$ (resp. $V(\Gamma')$) the number of vertices in $\Gamma$ (resp. $\Gamma'$), and denote by $d$ the cardinality of $\mathcal{L}$. Suppose that $B$ is not a Brauer graph algebra. Then $d>0$ and $\Gamma'$ is not bipartite. According to Theorem \ref{thm-maximal-tori}, the rank of the maximal torus of $\mathrm{Out}^{0}(A)$ (resp. $\mathrm{Out}^{0}(B)$) is $E(\Gamma)-V(\Gamma)+2$ (resp. $E(\Gamma')-V(\Gamma')-d+2$). Since $A$ and $B$ are stably equivalent of Morita type, $\mathrm{Out}^{0}(A)$ and $\mathrm{Out}^{0}(B)$ are isomorphic as algebraic groups, and therefore the ranks of the maximal tori of $\mathrm{Out}^{0}(A)$ and $\mathrm{Out}^{0}(B)$ are equal. Since $E(\Gamma)=E(\Gamma')$, we have that $V(\Gamma)=V(\Gamma')+d$. By \cite[Proposition 3.1]{A}, the rank of the Cartan matrix $C(B)$ of $B$ is $V(\Gamma')$, and the rank of the Cartan matrix $C(A)$ of $A$ is equal to $V(\Gamma)$ (if $\Gamma$ is not bipartite) or $V(\Gamma)-1$ (if $\Gamma$ is bipartite). Since $A$ and $B$ have the same number of isomorphism classes of simple modules, by Lemma \ref{lem-invariant-factors}, $C(A)$ and $C(B)$ have the same rank. Then $V(\Gamma')+d=V(\Gamma)\leq \mathrm{rank}(C(A))+1=\mathrm{rank}(C(B))+1=V(\Gamma')+1$. Therefore $d=1$ and $V(\Gamma)=\mathrm{rank}(C(A))+1$, which implies that $\Gamma$ is bipartite. In particular, $\Gamma$ has no faces of perimeter $1$.
   
Denote by $l_i$ (resp. $l'_i$) the number of faces of $\Gamma$ (resp. $\Gamma'$) of perimeter $i$, $r_i$ (resp. $r'_i$) the number of vertices of $\Gamma$ (resp. $\Gamma'$) of multiplicity $i$. By Remark \ref{remark-center-quotient}, we have $r_i=r'_i$ for $i>2$ and $r_2=r'_2+l'_1$. According to Corollary \ref{cor-stbA-tube-and-face}, $\Gamma$ and $\Gamma'$ have the same multi-set of perimeters of faces of perimeter larger than $2$. Moreover, since the sum of perimeters of all faces of $\Gamma$ and $\Gamma'$ are equal, we imply that $l_1+2l_2=l'_1+2l'_2$. Since $l_1=0$ and $l'_1\neq 0$, we have $l'_1\geq 2$.

Let $N=V(\Gamma')$, $m_1,m_2,\cdots,m_{N+1}$ (resp. $m'_1,m'_2,\cdots,m'_N$) be the multiplicities of all vertices of $\Gamma$ (resp. $\Gamma'$). According to \cite[Corollary 4.5]{A09}, the greatest common divisor of all $N\times N$ minors of $C(A)$ is $m_1 m_2 \cdots m_{N+1}(\frac{1}{m_1}+\frac{1}{m_2}+\cdots\frac{1}{m_{N+1}})$. Moreover, according to \cite[Corollaries 2.6 and 5.7]{A09}, the greatest common divisor of all $N\times N$ minors of $C(B)$ is $4m'_1 m'_2 \cdots m'_N$. Since
$$m_1 m_2 \cdots m_{N+1}=\prod_{i\geq 1}i^{r_i}=2^{l'_1}\prod_{i\geq 1}i^{r'_i}=2^{l'_1}m'_1 m'_2 \cdots m'_N,$$
and since $C(A)$ and $C(B)$ have the same determinantal divisors, we have
$$2^{l'_1}(\frac{1}{m_1}+\frac{1}{m_2}+\cdots\frac{1}{m_{N+1}})=4.$$
Since the number of vertices in $\Gamma$ of multiplicity $2$ is $r_2=r'_2+l'_1\geq 2$, $$\frac{1}{m_1}+\frac{1}{m_2}+\cdots\frac{1}{m_{N+1}}=\sum_{i\geq 1}\frac{r_i}{i}\geq \frac{r_2}{2}\geq 1.$$ Then $2^{l'_1}\leq 4$ and $l'_1\leq 2$. Note that we also have $l'_1\geq 2$, which implies that $l'_1=2$ and $\frac{1}{m_1}+\frac{1}{m_2}+\cdots\frac{1}{m_{N+1}}=1$. Since $$\frac{1}{m_1}+\frac{1}{m_2}+\cdots\frac{1}{m_{N+1}}=\sum_{i\geq 1}\frac{r_i}{i}\geq \frac{r_2}{2}=\frac{r'_2+l'_1}{2}=\frac{r'_2+2}{2},$$
we have $r'_2=0$, $r_2=r'_2+l'_1=2$, and $r_i=0$ for $i\neq 2$. So $V(\Gamma)=2$ and $V(\Gamma')=V(\Gamma)-1=1$. Since $r'_2=0$ and $r'_i=r_i=0$ for $i>2$, $\Gamma'$ has no vertices of multiplicity $\geq 2$, and the unique vertex of $\Gamma'$ has multiplicity $1$.

Suppose that $E(\Gamma)=E(\Gamma')=n$. Then $\Gamma'$ is a Brauer graph with one vertex and $n$ loops, where the multiplicity of the unique vertex is $1$. Since $\Gamma$ is bipartite, it contains no loops. Then $\Gamma$ is a Brauer graph with two vertices $v_1$, $v_2$, each has multiplicity $2$, together with $n$ edges connecting $v_1$ and $v_2$. A calculation shows that both $C(A)$ and $C(B)$ are $n\times n$ matrices with all entries equal to $4$. Then both $C(A)$ and $C(B)$ are zero matrices over $k$, and the projective centers $Z^{\mathrm{pr}}(A)$, $Z^{\mathrm{pr}}(B)$ are equal to zero. By Proposition \ref{prop:P} the centers of $A$ and $B$ are isomorphic as $k$-algebras. By Remark \ref{remark-center-of-stBA} and Remark \ref{remark-center-of-SSBA}, $Z(A)/(\mathrm{soc}(Z(A)))\cong k[X,Y]/(X^{2},Y^{2},XY)$ and $Z(B)/(\mathrm{soc}(Z(B)))\cong k$, which are not isomorphic, a contradiction.
\end{proof}

\end{document}